\documentclass[12pt, reqno, a4paper,oneside]{amsart}
\usepackage{graphicx, epsfig, psfrag}
\usepackage{amsfonts,amssymb,amsbsy,amsthm,amsmath}
\usepackage[T1]{fontenc}
\usepackage{bm}
\usepackage{color}
\usepackage{float}
\usepackage{hyperref}
\usepackage{mathrsfs}
\usepackage{subfigure}
\usepackage{longtable}
\usepackage{graphicx}
\usepackage{bbm}
\usepackage{mathtools}

\usepackage{ulem}
\usepackage[top=2cm,bottom=2cm,left=1.5cm,right=1.5cm]{geometry}

\usepackage{environments}
\numberwithin{theorem}{section}
\numberwithin{equation}{section}

\renewcommand{\cases}[1]{\left\{ \begin{array}{rl} #1 \end{array} \right.}
\newcommand{\smfrac}[2]{{\textstyle \frac{#1}{#2}}}

\def\R{\mathbb{R}}
\def\C{\mathbb{C}}
\def\N{\mathbb{N}}
\def\Z{\mathbb{Z}}

\def\LL{\mathrm{L}}

\def\CC{\mathrm{C}}
\def\HH{\mathrm{H}}

\def\<{\langle}
\def\>{\rangle}

\def\wto{\rightharpoonup}

\def\b{\big}
\def\B{\Big}
\def\bg{\bigg}
\def\Bg{\Bigg}

\def\Us{\mathscr{W}}
\def\Usz{\mathscr{W}_0}
\def\Hsi{\dot{\mathscr{W}}^{1,2}}

\def\Rg{\mathcal{R}}

\def\dist{\mathrm{dist}}

\def\mR{{\sf R}}

\def\sep{\,|\,}
\def\bsep{\,\b|\,}
\def\Bsep{\,\B|\,}

\def\del{\delta}
\def\ddel{\delta^2}

\def\dx{\,{\rm d}x}

\def\dt{\,{\rm d}t}

\def\ds{\,{\rm d}s}

\def\Aint#1{\mathchoice
  {\AXint\displaystyle\textstyle{#1}}%
  {\AXint\textstyle\scriptstyle{#1}}%
  {\AXint\scriptstyle\scriptscriptstyle{#1}}%
  {\AXint\scriptscriptstyle\scriptscriptstyle{#1}}%
  \!\int}
\def\AXint#1#2#3{{\setbox0=\hbox{$#1{#2#3}{\int}$}
\vcenter{\hbox{$#2#3$}}\kern-.5\wd0}}
\def\avint{\Aint-}

\def\ld{\lambda_d}
\def\ldR{\lambda_{d,R}}
\def\lLR{\lambda_{L,R}}

\def\L{\Lambda}
\def\yh{\hat{y}}
\def\alh{\hat{\alpha}}

\def\ycorr{\bar{y}}
\def\Bonds{\mathcal{B}}
\def\Cells{\mathcal{C}}
\def\Om{\Omega}
\def\BOm{\Bonds^\Omega}
\def\COm{\Cells^\Omega}
\def\Cores{\mathcal{C}^\pm}

\def\ind{{\rm index}}

\def\D{\mathcal{D}}

\def\Om{\Omega}

\def\TL{\mathcal{T}_\L}

\def\eps{\epsilon}

\DeclareMathOperator*{\argmin}{{\rm argmin}}

\def\conv{{\rm conv}}

\def\supp{{\rm supp}}

\def\psilin{\psi_{\mathrm{lin}}}

\def\diam{\mathrm{diam}}

\title[Stable screw dislocation configurations]{Analysis of stable screw dislocation configurations in an anti--plane lattice model}

\author{T. Hudson}
\address{T. Hudson \\ Mathematical Institute \\ University of Oxford \\
  Oxford OX1 3LB \\ UK}
\email{thomas.hudson@maths.ox.ac.uk}

\author{C. Ortner}
\address{C. Ortner\\ Mathematics Institute \\ Zeeman Building \\
  University of Warwick \\ Coventry CV4 7AL \\ UK}
\email{c.ortner@warwick.ac.uk}

\date{\today}

\thanks{CO was supported by the EPSRC grant EP/H003096 ``Analysis of
  atomistic-to-continuum coupling methods''. TH was supported by the
  UK EPSRC Science and Innovation award to the Oxford Centre for
  Nonlinear PDE (EP/E035027/1).}



\begin{document}

\begin{abstract}
  We consider a variational anti-plane lattice model and demonstrate that at zero temperature, there exist locally stable states containing screw dislocations, given conditions on the distance between the dislocations and on the distance between dislocations and the boundary of the crystal. In proving our results, we introduce approximate solutions which are taken from the theory of dislocations in linear elasticity, and use the inverse function theorem to show that local minimisers lie near them. This gives credence to the commonly held intuition that linear elasticity is essentially correct up to a few spacings from the dislocation core.
\end{abstract}
\maketitle

\section{Introduction}

Plasticity in crystalline materials is a highly complex phenomenon, a key 
aspect of which is the movement of dislocations. Dislocations are line
defects within the crystal structure which were first hypothesised to act as
carriers of plastic flow in \cite{Orowan34,Polanyi34,Taylor34}, and later
experimentally observed in \cite{Bollmann56,HirschHornWhelan56}.

As they move through a crystal, dislocations interact with themselves
and other defects via the orientation-dependent stress fields they
induce \cite{HirthLothe}. This leads to complex coupled
behaviour, and efforts to create accurate mathematical models to
describe their motion and interaction, and so better engineer such
materials are ongoing (see for example \cite{BulatovCai,DislDyn}).

Over the last decade, a body of mathematical analysis of
dislocation models has begun to develop which aims to derive models
of crystal plasticity in a consistent way from models of dislocation
motion and energetics. Broadly, this work starts from either atomistic
models, as in \cite{ADLGP13,EHIM09,Ponsiglione07,ArizaOrtiz05}, or
`semidiscrete' models, where dislocations are lines or points in an
elastic continuum, as in \cite{GPPS12,MP12,SZ10,GLP10,VCOA07,GM06,GM05}.

In the present work we focus on the analysis of dislocations at the
atomistic level and therefore briefly recount recent achievements in
this area. In \cite{EHIM09} the focus is the derivation of homogenised
dynamical equations for dislocations and dislocation densities from a
generalisation of the Frenkel--Kontorova model for edge
dislocations. In \cite{ArizaOrtiz05} a clear mathematical framework
for describing the Burgers vector of dislocations in lattices was
developed, and the asymptotic form of a discrete energy is given in
the regime where dislocations are far from each other relative to the
lattice spacing.  In \cite{Ponsiglione07} a rigorous asymptotic
description of the energy in a finite crystal undergoing anti--plane
deformation with screw dislocations present is developed.
\cite{ADLGP13} follows in the same vein, broadening the class of
models considered, and also treating the asymptotics of a minimising
movement of the dislocation energy. In a similar anti--plane setting,
but in an infinite crystal, \cite{HudsonOrtner13} demonstrated that
there are globally stable states with unit Burgers vector.

In the present contribution we demonstrate the existence of
\emph{locally} stable states containing multiple dislocations with
arbitrary combinations of Burgers vector. Once more, our analysis
concerns crystals under anti--plane deformation, but in addition to the
full lattice, we now consider finite convex domains with
boundaries. Recent results contained in \cite{ADLGP13} also address
the question of local stability of dislocation configurations in
finite domains, but under different assumptions to those employed
here, and using a different set of analytical techniques. In
particular, our analysis employs discrete regularity results which
enable us to provide quantitative estimates on the equilibrium
configurations, while previous results only provide estimates on the
energies.

\subsection{Outline}
The setting for our results is similar to that described in
\cite{HudsonOrtner13}: our starting point is the energy difference
functional
\begin{equation*}
  E^\Om(y;\tilde{y}):=\sum_{b\in\BOm}\b[\psi(Dy_b)-\psi(D\tilde{y}_b)\b],
\end{equation*}
where $\Om \subset \L$ is a subset of a Bravais lattice, $\BOm$ is a
set of pairs of interacting (lines of) atoms, $Dy_b$ is a finite
difference, and $\psi$ is a 1-periodic potential.

We call a deformation $y$ a \emph{locally stable equilibrium} if $u=0$
minimises $E^\Om(y+u;y)$ among all perturbations $u$ which have finite energy,
and are sufficiently small in the energy norm.
The key assumption upon which we base our analysis is the existence of a
local equilibrium in the homogeneous infinite lattice containing a
dislocation which satisfies a condition which we term \emph{strong stability}
--- this notion is made precise in \S\ref{sec:stab_assump}.

Under this key assumption, our main result is Theorem
\ref{th:full_latt}.  This states that, given a number of positive and
negative screw dislocations, there exist locally stable equilibria
containing these dislocations in a given domain as long as the core
positions satisfy a minimum separation criterion from each other and
from the boundary of the domain. Furthermore, these configurations may
only be globally stable if there is one dislocation in an infinite
lattice.

The proof of Theorem \ref{th:full_latt} is divided into two cases, that in 
which $\Om=\L$, and that in which $\Om$ is a finite convex lattice polygon:
these are proved in \S\ref{sec:inf} and \S\ref{sec:fin}
respectively.

\section{Preliminaries}
\label{sec:prelims}
\subsection{The lattice}
\label{sec:latt_compl}
Underlying the results presented in this paper is the structure of the
triangular lattice
\begin{equation}
  \L:=\smfrac{a_1+a_2}{3}+[a_1,a_2]\cdot\Z^2,\quad\text{where }
   a_1=(1,0)^T\text{ and }a_2=\b(\smfrac12,\smfrac{\sqrt3}{2}\b)^T.
   \label{eq:latt_defn}
\end{equation}
In this section we detail the main geometric and topological
definitions we use to conduct the analysis.

\subsubsection{The lattice complex}
For the purposes of providing a clear definition of the notion of Burgers
vector in the lattice, we describe the construction of a CW
complex\footnote{For further details on the definition of a CW complex and
other aspects of algebraic topology, see for example \cite{Hatcher}.} for a
general lattice subset. Recall from \cite{ArizaOrtiz05} that we may define a
lattice complex in 2D by first defining a set of lattice points (or
0--cells), $\L$, then defining the bonds (or 1--cells), $\Bonds$, and finally
the cells (or 2--cells) $\Cells$, and the corresponding boundary operators, 
$\partial$. Throughout the paper, $\L$, $\Bonds$ and $\Cells$ will refer to
the lattice complex generated by $\L$ as defined in \eqref{eq:latt_defn} ---
see also \cite{HudsonOrtner13} for further details of this construction.
Here, we also consider subcomplexes generated by subsets $\Om\subset\L$.

Given $\Om\subseteq\L$, we define the corresponding sets of bonds and cells to
be
\begin{align*}
  \BOm :=\b\{(\xi,\zeta)\in\Bonds\bsep \xi,\zeta\in\Om\b\} \quad
  \text{and} \quad
  \COm :=\b\{(\xi,\zeta,\eta)\in\Cells\bsep \xi,\zeta,\eta\in\Om\b\}.
\end{align*}
It is straightforward to check that this satisfies the definitions of a 
CW subcomplex of
the full lattice complex presented in \cite[\S2.3]{HudsonOrtner13},
and so we may make use of the definitions of integration and $p$-forms as
given in \cite[\S3]{ArizaOrtiz05} restricted to this subcomplex. To keep
notation concise, we will frequently write
\begin{equation}
  f_b:=f(b)\qquad\text{when}\quad f:\Bonds\to\R\quad\text{is a 1--form.}
\end{equation}

We note that we have chosen to define $\L$ such that $0\in\R^2$ lies at the
barycentre of a cell which we will denote $C_0$, and more generally we will
use the notation $x^C\in\R^2$ to refer to the barycentre of $C\in\Cells$.

\subsubsection{Lattice symmetries}
\label{sec:latt_auto}
The triangular lattice is a highly symmetric structure, and all of its
rotational symmetries can be described in terms of multiples of positive
rotations by $\pi/3$ about various points in $\R^2$. We therefore fix $\mR_6$
to be the corresponding linear transformation.

We define two special classes of affine transformations on $\R^2$ which are
automorphisms of $\L$,
\begin{align*}
  G^C(x)&:=\mR_6^i(x-x^C)=0\quad\text{where }i\in\{0,1\}
    \text{ is chosen so that }G^C(\L)=\L,\\
  H^C(x)&:=\b(G^C\b)^{-1}(x)=\mR_6^{-i}x+x^C,
\end{align*}
and note that by definition, $G^C(C)=C_0$ and $H^C(C_0)=C$.
We also understand $G^C,H^C$ as automorphisms on $\Bonds$ and $\Cells$ in the
following way: if $b=(\xi,\zeta)\in\Bonds$ and $C'=(\xi,\zeta,\eta)\in\Cells$,
then
\begin{equation*}
  G^C(b):=\b(G^C(\xi),G^C(\zeta)\b),\quad\text{and}\quad
  G^C(C'):=\b(G^C(\xi),G^C(\zeta),G^C(\eta)\b).
\end{equation*}
Later, it will be important to consider the transformation of 1--forms
under such automorphisms, and so we write
\begin{equation*}
  \b(f\circ G^C\b)_b:=f\b(G^C(b)\b)\qquad\text{when}\quad 
     f:\Bonds\to\R\quad\text{is a 1--form.}
\end{equation*}

\subsubsection{Nearest neighbours}
We define the set of nearest neighbour directions by
\begin{equation*}
  \Rg:=\{a_i\in\R^2\sep i\in\Z\},\quad\text{where}\quad a_i := \mR_6^{i-1}a_1.
\end{equation*}
Given $\Om\subseteq\L$ and $\xi\in\Om$, we define the nearest neighbour
directions of $\xi$ in $\Om$ to be
\begin{equation*}
  \Rg_\xi^\Om:=\b\{a_i\in\Rg\bsep\xi+a_i\in\Om\}\subseteq\Rg.
\end{equation*}

\subsubsection{Distance}
To describe the distance between
elements in the complex, we use the usual notion of Euclidean distance of sets,
\begin{equation*}
  \dist(A,B):=\inf\b\{|x-y|\bsep x\in A,\,y\in B\b\}.
\end{equation*}

\subsection{Convex crystal domains}
\label{sec:conv_crystals}
In addition to studying dislocations in the infinite lattice $\L$, we
will also consider dislocations in a {\em convex lattice polygon}: We
say that $\Om \subset \L$ is a {\em convex lattice polygon} if
\begin{displaymath}
  C_0 \in \COm, \quad \conv(\Om)\cap\L=\Om,
  \quad \text{and} \quad \Om \text{ is finite.}
\end{displaymath}
Here and throughout the paper, $\conv(U)$ means the closed convex hull of
$U\subset\R^2$, and $\Om$ will denote either a convex lattice polygon or
$\L$ unless stated otherwise. For a convex lattice polygon, we define
corresponding `continuum' domains
\begin{equation*}
  U^\Om :=\conv(\Om)\qquad\text{and}\qquad W^\Om:=\mathrm{clos}\B(\bigcup\b\{C\in\COm\bsep C\text{ positively--oriented}\b\}\B).
\end{equation*}
We note that $\Om\subset W^\Om\subseteq U^\Om$; for an illustration of an
example of these definitions, see Figure \ref{fig:conv_poly}.

\subsubsection{Boundary and boundary index}
\label{sec:bdry_index}
%
%
We note that the positively--oriented boundary $\partial W^\Om$ may be
decomposed as
\begin{equation*}
  \partial W^\Om = \b\{\xi\in\Om\sep \Rg_\xi^\Om\neq\Rg\b\}\cup\B\{b\in\BOm
   \Bsep b\in\partial\sum\{C\in\COm\sep C\text{ positively--oriented}\}\B\};
\end{equation*}
in other words, into the lattice points which do not have a full set of
nearest neighbours, and hence lie on the `edge' of the set $\Om$, and into
the bonds which follow the positively--oriented boundary of the entire
subcomplex within the full lattice. Since it will be necessary to sum over
these sets later, we write $\xi\in\partial W^\Om$ to mean
$\xi\in\partial W^\Om\cap\Om$, and $b\in\partial W^\Om$ to mean
$b\in\partial C\cap\partial W^\Om$ for some positively--oriented $C\in\COm$.

It is clear that since $\Om$ is a finite set, $U^\Om$ is a convex polygonal
domain in $\R^2$, and $\partial U^\Om$ is made up of finitely--many straight
segments.
We number the corners of such polygons according to the positive orientation
of $\partial U^\Om$ as $\kappa_m$, $m= 1, \dots, M$, and $\kappa_0:=\kappa_M$;
evidently, $\kappa_m\in\Om$ for all $m$. We further set
$\Gamma_m:=(\kappa_{m-1},\kappa_m)\subset\R^2$, the straight segments of
the boundary.

For each $m$, $\kappa_m-\kappa_{m-1}$ is a lattice direction. Since any pair
$a_i,a_{i+1}$ with $i\in\Z$ form a basis for the lattice directions, there
exists $i$ such that
\begin{equation*}
  \kappa_m-\kappa_{m-1} = j' a_i + k' a_{i+1},\quad\text{where}\quad j',
  k' \in \N, j' > 0.
\end{equation*}
Define the \emph{lattice tangent vector} to $\Gamma_m$ to be
\begin{equation*}
  \tau_m:=j a_i+ k a_{i+1},\quad\text{where }\gcd(j,k)=1\quad\text{and}\quad
    \kappa_m-\kappa_{m-1} = n \tau_m\quad\text{for some }n \in\N.
\end{equation*}
This definition entails that $\tau_m$ is irreducible in the sense that no
lattice direction with smaller norm is parallel to $\tau_m$, and hence if
$\zeta\in\Gamma_m\cap\partial W^\Om$, $\zeta=\kappa_m+j\tau_m$ for some
$j=\{0,\ldots J_m\}$. In addition to the decomposition of $\partial W^\Om$
into lattice points and bonds, we may also decompose into `periods'
$P_\zeta$ indexed by $\zeta\in\partial W^\Om\cap\partial U^\Om$, so
\begin{equation}
  \partial W^\Om = \bigcup_{m=1}^M\bigcup_{j=0}^{J_m}P_{\kappa_m+j\tau_m}
  \qquad\text{where}\quad P_{\zeta}:=\b\{x\in\partial W^\Om\bsep (x-\zeta)\cdot\tau_m\in\b[0,|\tau_m|^2\b]\b\}.\label{eq:Pzeta_defn}
\end{equation}
An illustration of the definition of $P_\zeta$ may be found on the 
right--hand side of Figure \ref{fig:conv_poly}.
Denoting the 1--dimensional Hausdorff measure as $\mathcal{H}^1$, we define
the {\em index} of $\Gamma_m$, and of $\partial W^\Om$ respectively, to be
\begin{equation}
  \ind(\Gamma_m) := \mathcal{H}^1(P_\zeta)\quad\text{for any }
    \zeta\in\Gamma_m\cap\partial W^\Om \quad \text{and}\quad
  \ind(\partial W^\Omega):= \max_{m = 1, \dots, M} \ind(\Gamma_m).
  \label{eq:ind_defn}
\end{equation}

\begin{figure}
  \includegraphics[width=0.43\textwidth]{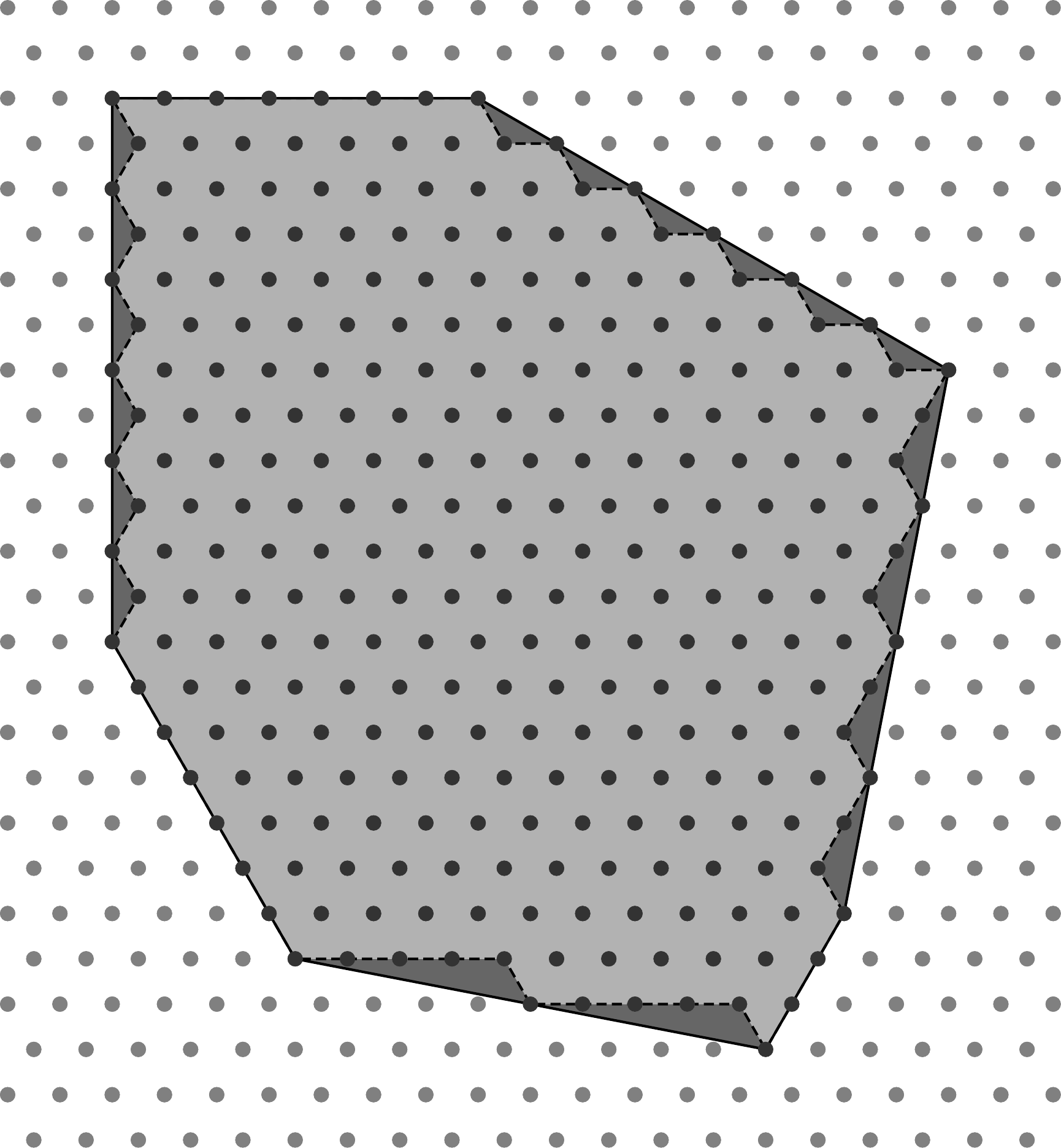}
  \hspace{0.03\textwidth}
  \includegraphics[width=0.43\textwidth]{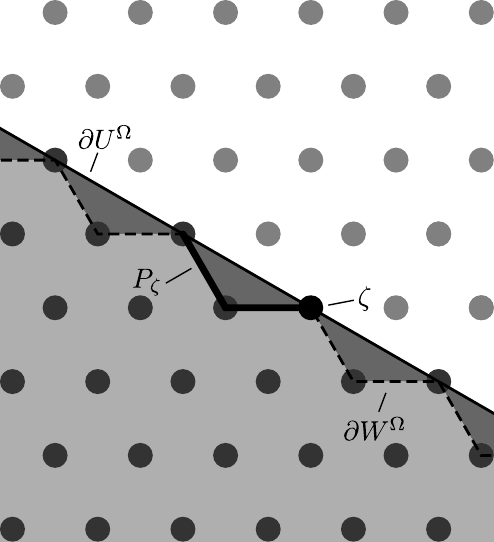}
  \caption{The figure on the left shows an example of a convex lattice 
  polygon. Here, $\Om$ is the set of dark grey points, $W^\Om$ is the light
  grey region and the dark grey region corresponds to $U^\Om\setminus W^\Om$.
  The boundaries of $W^\Om$ and $U^\Om$ are denoted by dashed and plain lines
  respectively.
  The figure on the right illustrates the definition of $P_\zeta$, clearly
  showing the periodic structure of $\partial W^\Om$.}
  \label{fig:conv_poly}
\end{figure}

\subsection{Deformations and Burgers vector}
The positions of deformed atoms will be described by maps
$y\in\Us(\Om):=\{y:\Om\to\R\}$. For $y\in\Us(\Om)$ and
$b=(\xi,\eta)\in\BOm$, we define the finite difference
\begin{equation*}
  Dy_b = y(\eta)-y(\xi).
\end{equation*}

\subsubsection{Function spaces}
In addition to the space $\Us(\Om)$, we define
\begin{align*}
  \Usz(\Om)&:=\b\{v\in\Us(\Om)\bsep v(\xi_0)=0\text{ and }
    \supp(Dv)\text{ is bounded.}\b\},\\
  \Hsi(\Om)&:=\b\{v\in\Us(\Om)\bsep v(\xi_0)=0\text{ and }
    Dv\in\ell^2\b(\Bonds^\Om\b)\b\},
\end{align*}
where $\xi_0=(0,\frac{\sqrt3}{3})^T\in\Om$. It is shown in
\cite[Prop. 9]{OrtSha:interp:2012} that $\Hsi$ is a Hilbert
space and $\Usz \subset \Hsi$ is dense.

\subsubsection{Burgers vector}
We now slightly generalise some key definitions from \cite{HudsonOrtner13}.

Given $y:\Om\to\R$, the set of associated {\em bond length 1-forms} is
defined to be
\begin{equation*}
  [Dy]:=\b\{\alpha:\BOm\to[-\smfrac12,\smfrac12]\bsep \alpha_{-b}=-\alpha_b,
   \,Dy_b-\alpha_b\in\Z\text{ for all }b\in\BOm,\,\alpha_b\in(-\smfrac12,\smfrac12]\text{ if }b\in\partial W^\Om\b\}.
\end{equation*}

A {\em dislocation core} of a bond length 1-form $\alpha$ is a
positively--oriented 2-cell $C\in\COm$ such that $\int_{\partial C}
\alpha \neq 0$.


As remarked in \cite[\S2.5]{HudsonOrtner13}, the Burgers vector of a single
cell may only be $-1$, $0$ or $1$, so we define the set of dislocation cores
to be
\begin{equation*}
  \Cores[\alpha]:=\B\{C\in\COm\Bsep C\text{ positively--oriented, }\int_{\partial C}\alpha=\pm1\B\}.
\end{equation*}
We can now define the {\em net Burgers vector} of a deformation $y$
with $|\Cores[\alpha]| < \infty$ (i.e., a finite number of dislocation
cores) to be
\begin{displaymath}
  B[y] := \sum_{C\in\Cores[\alpha]} \int_{\partial C} \alpha.
\end{displaymath}
If $\Omega$ is a convex lattice polygon, then it is straightforward to
see that $B[y] = \int_{\partial W^\Om}\alpha$, and the requirement that
$\alpha_b\in(-\smfrac12,\smfrac12]$ for $b\in\partial W^\Om$ ensures that
the net Burgers vector is independent of $\alpha\in[Dy]$, since any two
bond length 1-forms agree for all $b\in\partial W^\Om$.




\subsection{Dislocation configurations}
\label{sec:core_config}
In order to prescribe the location of an array of dislocations, we
define a \emph{dislocation configuration} (or simply, a
\emph{configuration}) to be a set $\D$ of ordered pairs
$(C,s)\in\COm\times\{-1,1\}$, satisfying the condition that
\begin{equation}
  (C,s)\in\D\quad\text{implies}\quad(C,-s)\notin\D.\label{eq:D_uniq_bvec}
\end{equation}
Such sets $\D$ should be thought of as a set of dislocation core positions
with accompanying Burgers vector $\pm1$.
We define the minimum separation distance of a configuration to be
\begin{equation*}
  L_\D:=\inf\b\{\dist(C,C')\bsep (C,s),(C',t)\in\D, C\neq C'\},
\end{equation*}
and in the case where $\Om$ is a convex lattice polygon, we define the
minimum separation between the dislocations and the boundary to be
\begin{equation*}
  S_\D:=\inf\b\{\dist(C,\partial W^\Om)\bsep(C,s)\in\D\b\}.
\end{equation*}



\section{Main results}
\label{sec:main_results}

\subsection{Energy difference functional and equilibria}
We assume that lattice sites interact via a $1$--periodic
nearest--neighbour pair potential $\psi \in C^4(\R)$, which is even
about $0$. We discuss possible extensions in \S\ref{sec:discussion}.

For a pair of displacements $y,\tilde{y}\in\Us(\Om)$, we define 
\begin{equation*}
  E^\Om(y;\tilde{y}):=\sum_{b\in\BOm}\psi(Dy_b)-\psi(D\tilde{y}_b),
\end{equation*}
where we will drop the use of the superscript in the case where $\Om=\L$.  We
note immediately that this functional is well-defined whenever
$y-\tilde{y}\in\Usz(\Om)$. It is also clear that Gateaux derivatives
in the first argument (in $\Usz(\Om)$ directions) exist up to fourth
order, and do not depend on the second argument. We denote these derivatives
$\del^j E^\Omega(y)$, so that for $v, w \in \Usz(\Omega)$, we have
\begin{align*}
  \<\del E^\Om(y),v\> :=\sum_{b\in\Bonds^\Om}\psi'(Dy_b)\cdot Dv_b,
  \quad \text{and} \quad 
  \<\ddel E^\Om(y)v,w\> :=\sum_{b\in\Bonds^\Om} \psi''(Dy_b)\cdot Dv_bDw_b.
\end{align*}

In \S\ref{sec:ancil:ediff}, we will demonstrate that under certain
conditions on $\tilde{y}$, $E(y;\tilde{y})$ it may be extended by continuity
in its first argument to a functional which is also well--defined for
$y-\tilde{y}\in\Hsi(\Om)$.

The following definition makes precise the various notions of equilibrium we
will consider below.

\begin{definition}[Stable Equilibrium]
  \label{def:defn_stable_equilib}
  (i) A displacement $y\in\Us(\Om)$ is a {\em locally stable equilibrium} if
  there exists $\epsilon > 0$ such that $E^\Om(y + u; y) \geq 0$ for
  all $u\in\Usz(\Om)$ with $\| D u \|_2 \leq \epsilon$.

  (ii) We call a locally stable equilibrium $y$ {\em strongly stable}
  if, in addition, there exists $\lambda > 0$ such that
  \begin{equation}
    \label{eq:defn_strong_stab}
    \< \ddel E^\Omega(y) v, v \> \geq \lambda \| Dv\|_{\ell^2}^2
    \qquad \forall v \in \Usz(\Om).
  \end{equation}

  (iii) A displacement $y\in\Us(\Om)$ is a {\em globally stable
    equilibrium} if $E^\Om(y + u; y) \geq 0$ for all $u \in
  \Usz(\Om)$.
\end{definition}

\subsection{Strong stability assumption}
\label{sec:stab_assump}
Here, we discuss the key assumption employed in proving the main
results of the paper. As motivation, we review a result from
\cite{HudsonOrtner13}. 

Let $\yh : \R^2 \setminus \{0\} \to \R$ be the dislocation solution
for anti--plane linearised elasticity \cite{HirthLothe}, i.e.
\begin{equation*}
  \yh(x):=\smfrac{1}{2\pi}\arg(x)
    =\smfrac{1}{2\pi}\arctan\b(\smfrac{x_2}{x_1}\b),
\end{equation*}
where we identify $x \in \R^2$ with the point $x_1+i x_2 \in \C$, and the
branch cut required to make this function single--valued is taken along
the positive $x_1$--axis. The accepted intuition is that $\yh$ provides
a good description of dislocation configurations, except in a `core'
region which stores a finite amount of energy. Reasonable candidates for
equilibrium configurations are therefore of the form $y = \yh + u$
where $u \in \Hsi(\Om)$. This intuition is made precise as follows
\cite[Theorem 4.5]{HudsonOrtner13}:

\begin{theorem}[Global stability of single dislocation $\L$]
  \label{th:global}
  In addition to the foregoing assumptions, suppose that $\psi(r) \geq
  \smfrac12\psi''(0) r^2$ for $r\in[-\smfrac12,\smfrac12]$ where
  $\psi''(0) > 0$, then there exists $u\in\Hsi(\L)$ such that $\yh+u$
  is a globally stable equilibrium of $E$.
\end{theorem}

\medskip

In the present work, where we focus on multiple dislocation cores, we
remove the additional technical assumptions on $\psi$ but instead directly
assume the existence of a single stable core; i.e.
\medskip

\begin{description}
\item[{\bf (STAB)}] there exists $u \in \Hsi(\L)$ such that $y=\yh+u$
  is a strongly stable equilibrium.
\end{description}
\medskip

Throughout the rest of the paper, $u$ is fixed to satisfy {\bf
  (STAB)}. We denote $\lambda_d := \lambda$ to be the stability
constant from (\ref{eq:defn_strong_stab}) with $y = \yh+u$, and
fix a finite collection of cells, $A$, such that
$\Cores[\alpha] \subset A$ for any $\alpha\in[D\yh+Du]$.

\begin{remark}
  To demonstrate that {\bf (STAB)} holds for a non-trivial class of
  potentials $\psi$ satisfying our assumptions, we refer to Lemma 3.3
  in \cite{HudsonOrtner13}, which states that, if $\psi=\psilin$,
  where
  \begin{equation*}
    \psilin(x):=\smfrac12\lambda\,\dist(x,\Z)^2,
  \end{equation*}
  then $\ddel E > 0$ at $y=\yh + u$, a globally stable
  equilibrium. (Theorem \ref{th:global} does not in fact require
  global smoothness of $\psi$.) Furthermore, it immediately follows
  that $\dist(Dy_b,\smfrac12+\Z)\geq \eps_0$ for some $\eps_0>0$.

  Using the inverse function theorem as stated below in Lemma
  \ref{th:inv_func_theorem}, it is fairly straightforward to see that
  if $\psi\in\CC^4(\R)$ satisfies
  \begin{equation*}
    \b| \psi^{(j)}(r) - \psilin^{(j)}(r) \b| \leq \epsilon |r|^{p-j}
    \text{ for } r \in [-1/2, 1/2] \text{ and } j = 1, 2,
  \end{equation*}
  where $p > 2$ and $\epsilon$ is sufficiently small, then there
  exists $w \in \Hsi(\L)$ with $\|Dw \|_{\ell^2} \leq C \epsilon$ such
  that $\yh + u + w$ is a strongly stable equilibrium for $E$.

  Potentials constructed in this way are by no means the only
  possibilities --- {\bf (STAB)} can in fact be checked for any given
  potential by way of a numerical calculation, using for example the
  methods analysed in \cite{EOS13}.
  
  We also remark here that in \S5 of \cite{ADLGP13}, a demonstration of
  the existence of local minimisers is given under different assumptions.
  Instead of {\bf (STAB)}, structural assumptions are made on the
  potential, which the example we provide here also satisfies. Under these
  assumptions, Lemma 5.1 in \cite{ADLGP13} demonstrates that there exist
  energy barriers which dislocations must overcome in order to move from
  cell to cell. This leads to the proof of Theorem 5.5, which includes the
  statement that there exist local minimisers containing dislocations in
  finite lattices.
\end{remark}



\subsection{Existence Results}
\label{sec:results_inf}
We state the existence result for stable dislocation configurations in
the infinite lattice and in convex lattice polygons together. To this
end, we denote $S_\D := +\infty$ for the case $\Om = \L$.

We remark here that the main achievement of this analysis is to show the
constants $L_0$ and $S_0$ depend only on the number of dislocations, the
potential $\psi$ and $\ind(\partial W^\Om)$, and not on the specific
domain $\Om$ or its diameter.

\begin{theorem}
  \label{th:full_latt}
  Suppose that {\bf (STAB)} holds and either $\Om = \L$ or $\Om$ is a convex
  lattice polygon.

  (1) For each $N \in \N$, there exist constants $L_0=L_0(N)$ and
  $S_0=S_0\b(\ind(\partial W^\Om),N\b)$ such that for any core configuration
  $\D$ satisfying $|\D| = N$, $L_\D \geq L_0$ and $S_\D \geq S_0$,
  there exists a strongly stable equilibrium $z \in \Us(\Om)$, and for
  any $\alpha\in[Dz]$,
  \begin{equation*}
    \Cores[\alpha] \subset \bigcup_{(C,s)\in\D} H^C(A) \qquad
    \text{and}\qquad \int_{\partial H^C(A)} \alpha = s,\,
    \text{ for all }(C,s)\in\D.
  \end{equation*}
  In particular, $B[z] = \sum_{(C,s)\in\D} s$, and the conditions
  $L_\D \geq L_0$ and $S_\D \geq S_0$ entail that core regions $x^C + A$ do
  not overlap each other, or with the boundary.

  (2) The equilibrium $z$ can be written as 
  \begin{displaymath}
    z = \sum_{(C, s) \in \D} s (\yh + u) \circ G^C + w,
  \end{displaymath}
  where $w \in \Hsi(\Om)$ and $\| Dw \|_{\ell^2} \leq c
  (L_\D^{-1} + S_\D^{-1/2})$, where $c$ is a constant depending only on $N$
  and $\ind(\partial W^\Om)$.


  (3) Unless $N = 1$ and $\Om = \L$, $z$ cannot be a globally
  stable equilibrium.
\end{theorem}





\subsubsection{Strategy of the proof}
\label{sec:proof_outline}
In both cases, the overall strategy of proof is similar:
\begin{itemize}
\item We construct an approximate equilibrium $z$, using the linear
  elasticity solution for the dislocation configuration and a truncated
  version of the core corrector whose existence we assumed in {\bf (STAB)}.
\item For given $\D$, we obtain bounds which demonstrate that
  $\del E^\Om(z)$ decays to zero as $L_\D, S_\D \to \infty$, where $z$ are
  the approximate equilibria corresponding to $\D$.
\item We show that $\ddel E^\Om(z) \geq \ld -\eps$ as $L_\D$,
  $S_\D\to\infty$.
\item We apply the Inverse Function Theorem to demonstrate the
  existence of a corrector $w \in \Hsi(\Om)$ such that $z+w$ is a
  strongly stable equilibrium. Since $\| Dw \|_{\ell^2}$ can be made
  arbitrarily small by making more stringent requirements on $\D$, we can
  demonstrate that the condition on the core position holds, which completes
  the proof of parts (1) and (2) of the statement.
\item Part (3) of the statement is proved by construction of
  explicit counterexamples.
\end{itemize}

\begin{remark}
  \label{rem:s12}
  The reduced rate $S_\D^{-1/2}$ (as opposed to $L_\D^{-1}$) with
  respect to separation from the boundary is due to surface stresses
  which are not captured by the standard linear elasticity theory that
  we use to construct the predictor $z$.

  By formulating a half-plane problem where
  $\partial W^\Om$ is not identical to $\partial U^\Om$,
  one may readily check that the bound is in general sharp. 

  If ${\rm index}(\partial W^\Om) = 1$, then it may be possible to
  prove that the rate should be bounded by $S_\D^{-1}$. Otherwise, it
  is necessary to add an additional boundary correction to the
  predictor which captures these surface stresses. All of these routes
  seem to require improved regularity estimates of the boundary
  corrector $\ycorr$ that we construct in~\eqref{eq:y_corr_prob}.
  %
\end{remark}

\subsubsection{Possible extensions}
\label{sec:discussion}
We have avoided the most difficult aspect of the analysis of
dislocations by imposing the strong stability assumption {\bf
(STAB)} for a single core. Once this is established (or assumed),
several extensions of our analysis become possible, which we disuss
in the following paragraphs.

{\it Symmetry of $\psi$: } An immediate extension is to drop the
requirement that $\psi$ is even about $0$, which would be the case if
the body was undergoing macroscopic shear. This extension would
require us to separately assume the existence of strongly stable
positive and negative dislocation cores in the full lattice, as they
would no longer necessarily be symmetric. Apart from the introduction
of logarithmic factors into some of the bounds we obtain, it appears
that the analysis would be analogous to that contained here.

{\it General domains: } It is straightforward to generalise the
analysis carried out in \S\ref{sec:fin} to `half--plane' lattices,
since the linear elastic corrector $\ycorr$ may be explicitly
constructed via a reflection principle.  This suggests that in fact
the analysis could be extended to hold in any convex domain with a
finite number of corners $\kappa_m\in\Om$ and tangent vectors $\tau_m$
which are lattice directions --- in effect, an `infinite' polygon. The
key technical ingredient required here would be to prove decay results
for the corrector problem analysed in \S\ref{sec:fin:corrector} in
such domains, which we were not able to find in the literature. It is
unclear to us, though, to what extent extensions to non-convex domains
are feasible.

{\it Interaction potential: } The assumption that interactions are
governed by nearest--neighbour pair potentials only is easily lifted
as well. A generalisation to many--body interactions with a finite
range beyond nearest neighbours is conceptually straightforward
(though would add some technical, and in particular notational
difficulties), as long as suitable symmetry assumptions are placed on
the many-body site potential. Note, in particular, that the crucial
decay estimates from \cite{EOS13} that we employ are still valid in
this case.


{\it In-plane models: } Generalisations to in--plane models seem to be
relatively straightforward only in the infinite-lattice case. In the
finite domain case, one would need to account for surface relaxation
effects, which we have entirely avoided here by choosing an
anti--plane model (however, see Remark \ref{rem:s12}. The phenomenon
of surface relaxation in discrete problems seems a difficult one, and
to the authors' knowledge, has yet to be addressed systematically in
the Applied Analysis literature, but for some results in this
direction, see \cite{Theil11}. A possible way forward would be to
impose an additional stability condition on the boundary, similar to
our condition {\bf (STAB)}, which could then be investigated
separately.

\section{Ancillary results}

\subsection{Extension of the energy difference functional}
\label{sec:ancil:ediff}
%
The following is a slight variation of \cite[Lemma 4.1]{HudsonOrtner13}.

\begin{lemma}
  Let $y\in\Us(\Om)$, and suppose that $\del E^\Om(y)$ is a bounded
  linear functional. Then $u\mapsto E^\Om(y+u;y)$ is continuous as a
  map from $\Usz(\Om)$ to $\R$ with respect to the norm
  $\|D\cdot\|_2$; hence there exists a unique continuous extension of
  $u \mapsto E^\Om(y+u;y)$ to a map defined on $\Hsi(\Om)$. The
  extended functional $u \mapsto E^\Om(y+u; y)$, $u \in \Hsi(\Om)$ is
  three time continuously Frechet differentiable.
\end{lemma}
\begin{proof}
  The proof of this statement is almost identical to the proof of
  \cite[Lemma 4.1]{HudsonOrtner13} and hence we omit it. We note that
  in a finite domain, the condition that $\del E^\Om(y)$ is a bounded
  linear functional is always satisfied, since $\Hsi(\Om)$ is a finite
  dimensional space.
\end{proof}

\subsection{Stability of the homogeneous lattice}
\label{sec:aux:stabhom}
The following lemma demonstrates that $y=0$ is a globally stable
as well as strongly stable equilibrium. In particular, this shows
that $\yh+u$ cannot be a unique stable equilibrium among all
$y\in\Us(\L)$.

\begin{lemma}
\label{th:0_glob_stab}
Suppose that ${\bf (STAB)}$ holds, then the deformation $y \equiv 0$
is a strongly stable equilibrium for any $\Om \subset \L$. Precisely,
\begin{displaymath}
  \<\ddel E^\Om(0)v,v\> = \psi''(0) \sum_{b \in \BOm}Dv_b^2 \qquad
  \text{and } \quad  \psi''(0) \geq \lambda_d.
\end{displaymath}
\end{lemma}
\begin{proof}
  Suppose that $v\in\Usz(\Om)$ and $C^i\in\Cells$ is a sequence such that
  $\dist(C^i,0)\to\infty$ as $i\to \infty$. Define $v^i:=v\circ G^{C^i}$; if
  $y=\yh+u$,
  \begin{equation*}
    \ld\|Dv\|_2^2=\ld\|Dv^i\|_2^2\leq\<\ddel E^\L(y)v^i,v^i\>=
      \sum_{b\in\Bonds} \psi''(Dy_b)\b(Dv^i_b\b)^2=
      \sum_{b\in\Bonds}\psi''\b(D(y\circ H^{C^i})_b\b)Dv^2_b,
  \end{equation*}
  and since $\dist(Dy_b,\Z)\to0$ as $\dist(b,0)\to\infty$, it follows that
  $0<\ld\leq\psi''(0)$.
  
  Since $\psi$ is even about $0$ it must be that $\psi'(0)=0$, and
  the statement follows trivially.
\end{proof}

\subsection{The linear elasticity residual}
\label{sec:linelast_res}
We now prove a result estimating the residual of the
pure linear elasticity predictor.

\begin{lemma}
  \label{th:residual_sum_yhat}
  Let $\D$ be a dislocation configuration in $\L$ and $z := \sum_{(C,
    s) \in \D} \yh \circ H^{C}$. 
  For $L_\D$ sufficiently large, there exists $g : \Bonds \to \R$ such
  that
  \begin{equation}
    \label{eq:residual_sum_yhat}
    \< \del E(z), v \> = \sum_{b \in \Bonds} g_b Dv_b
    \quad \text{and} \quad |g_b| \leq 
    c \sum_{(C,s) \in \D} \dist(b, C)^{-3}.
  \end{equation}
\end{lemma}
\begin{proof}
  The canonical form for $\del E(z)$ is $\< \del E(z), v \> = \sum
  \psi'(\alpha_b) Dv_b$, where $\alpha$ is a bond length one-form
  associated with $Dz$.  For $L_\D$ sufficiently large, arguing as in
  \cite[Lemma 4.3]{HudsonOrtner13} we obtain that $\alpha_b \in (-1/2,
  1/2)$, which entails that $\alpha\in[Dz]$ is unique, and may be written in
  the form $\alpha_{(\xi,\xi+a_i)} = \int_0^1 \nabla_{a_i} z(\xi+t a_i)\dt$,
  where here and below $\nabla z$ will mean the extension of the gradient
  of $z$ to a function in
  $\CC^\infty\b(\R^2\setminus \bigcup_{(C,s)\in\D}\{x^C\};\R^2\b)$.

  We note that $|\psi'(\alpha_b)| \lesssim \sum \dist(b,C)^{-1}$ only, so we
  must remove a ``divergence-free component''. To that end, let
  $\omega_b := \bigcup \{ C' \in \Cells \sep \pm b \in \partial C', C'
  \text{ positively--oriented}\}$ and let $V := |\omega_b|$ for some
  arbitrary $b \in \Bonds$. Further, let
  $\bar{C}_\epsilon := \bigcup_{(C,s) \in\D} B_\epsilon(x^C)$. Then, for
  $b = (\xi, \xi+a_i)$, we define
  \begin{displaymath}
    h_b :=  \frac{\psi''(0)}{V} \lim_{\epsilon \to 0} 
    \int_{\omega_b \setminus \bar{C}_\epsilon}  \nabla z \cdot a_i \dx
    \qquad \text{and} \qquad 
    g_b := \psi'(\alpha_b) - h_b.
  \end{displaymath}
  It is fairly straightforward to show that the limit exists by applying
  the divergence theorem, which entails that $h_b$ and $g_b$ are
  well--defined and
  \begin{displaymath}
    \sum_{b \in \Bonds} h_b Dv_b = 
    \lim_{\epsilon \to 0} \frac{\psi''(0)}{V} \int_{\R^2 \setminus \bar{C}_\epsilon}
    \nabla z \cdot \nabla Iv \dx = 0
  \end{displaymath}
  for all $v \in \Usz(\L)$, where $Iv$ denotes the continuous and
  piecewise affine interpolant of $v$.  Thus, we obtain that $\< \del
  E(z), v \> = \sum_{b \in \Bonds} g_b Dv_b$ as desired.

  It remains to prove the estimate on $g_b$. Taylor expanding, we obtain
  \begin{equation*}
    \psi'(\alpha_b)-h_b = \psi'(0)+\psi''(0)\bg(\alpha_b-\frac{1}{V}
      \lim_{\epsilon \to 0}\int_{\omega_b \setminus \bar{C}_\epsilon}
      \nabla z \cdot a_i \dx\bg)+\smfrac12\psi'''(0)|\alpha_b|^2
      +O\b(|\alpha_b|^3\b).
  \end{equation*}
  The first and third terms vanish since $\psi$ is even. Note that
  $\alpha_b=\frac1V\int_{\omega_b} \nabla z\cdot a_i\dx$, where $a_i$ is the
  direction of the bond $b$, so Taylor expanding about the midpoint of
  $b$ and using the symmetry of $b$ and $\omega_b$ to eliminate the term
  involving $\nabla^2z$, we obtain
  \begin{equation*}
    \int_b\nabla z\cdot a_i\dx-\frac{1}{V}
      \lim_{\epsilon \to 0}\int_{\omega_b \setminus \bar{C}_\epsilon}
      \nabla z \cdot a_i \dx = O\b(|\nabla^3z|\b).
  \end{equation*}
  Finally, as $|\alpha_b|\lesssim \dist(b,C)^{-1}$ and
  $|\nabla^3 z|\lesssim \dist(b,C)^{-3}$ for all $(C,s)\in\D$,
  the stated estimate follows.
\end{proof}

\subsection{Regularity of the corrector}
\label{sec:reg_corrector}
We now slightly refine the general regularity result of Theorem~3.1 in
\cite{EOS13}, exploiting the evenness of the potential $\psi$.

\begin{lemma}
  \label{th:reg_corrector}
  Let $u$ be the core corrector whose existence postulated in {\bf
    (STAB)}: then there exists a constant $C_{\rm reg}$ such that
  \begin{displaymath}
    |Du_b| \leq C_{\rm reg}\,\dist(b,C)^{-2}\qquad \text{ for all } b \in \Bonds\text{ and }(C,s)\in\D.
  \end{displaymath}
\end{lemma}
\begin{proof}
  Our setting satisfies all assumptions of the $d = 2, m = 1$
  (anti--plane) case described in Section 2.1 of \cite{EOS13} with
  $\mathcal{N}_\xi = \{ a_i \sep i = 1, \dots, 6 \}$ for all $\xi \in
  \L$, and the complete set of assumptions summarized in {\bf (pD)} in
  Section 2.4.5 of \cite{EOS13}. Using Lemma \ref{th:residual_sum_yhat},
  we may apply Lemma 3.4 \cite{EOS13} with $p=3$, implying
  $|Du_b| \lesssim \dist(b,C)^{-2}$.
\end{proof}

\subsection{Approximation by truncation}
Following \cite{EOS13} we define a family of truncation operators
$\Pi^C_R$, which we will apply to $u\in\Hsi(\Om)$.  Let
$\eta\in\CC^1(\R^2)$ be a cut off function which satisfies
\begin{equation*}
  \eta(x):=\cases{
  1, & |x|\leq\smfrac34,\\
  0, & |x|\geq1.
  }
\end{equation*}
Let $Iu$ be the piecewise affine interpolant of $u$ over the
triangulation given by $\TL=\Cells$. For $R > 2$ let
$A_R:=B_R\setminus B_{R/2+1}$, an annulus over which $\eta(x/R)$ is
not constant. Define $\Pi^C_R:\Hsi\to\Usz$ by
\begin{equation*}
  \Pi^C_Ru(\xi):=\eta\B(\smfrac{\xi-x^C}R\B)\b(u(\xi)-a^C_R\b),\qquad
    \text{where}\qquad a^C_R:=\avint_{x^C+A_R} Iu(x)\dx.
\end{equation*}
In addition, we define $\Pi_R := \Pi_R^{C_0}$.

We now state the following result concerning the approximation
property of the family of truncation operators $\Pi^C_R$, which
follows from results in \cite{EOS13}.

\begin{lemma}
  \label{th:trunc_est}
  Let $v \in \Hsi(\L)$ and $C \in \Cells$, then
  \begin{equation}
    \label{eq:trunc_est:general}
    \b\|D \Pi_R^C v - D v \b\|_{\ell^2(\Bonds)}\leq
    \gamma_1 \|Dv\|_{\ell^2(\Bonds\setminus B_{R/2}(x^C))}.
  \end{equation}
  where $\gamma_1$ is independent of $R, v$ and $C$. 

  In particular, if $u \in \Hsi(\L)$ is the core corrector from ${\bf
    (STAB)}$, then
  \begin{equation}
    \label{eq:trunc_est:core}    
    \b\|D\Pi^C_R (u \circ G^C) - D (u \circ G^C) \b\|_{\ell^2(\Bonds)} 
    \leq \gamma_2 R^{-1},
  \end{equation}
  where $\gamma_2$ is independent of $R$ and $C$.
\end{lemma}
\begin{proof}
  Since $\|\cdot\|_{\ell^2(\Bonds)}$ is invariant under composition of
  functions with lattice automorphisms we can assume, without loss of
  generality, that $C = C_0$. The estimate
  (\ref{eq:trunc_est:general}) then is simply a restatement of
  \cite[Lemma 4.3]{EOS13}. The second estimate
  (\ref{eq:trunc_est:core}) then follows immediately from Lemma
  \ref{th:reg_corrector}.
  %
\end{proof}

Next, we show that the assumption {\bf (STAB)} implies that that
$\ddel E^\L(\yh+\Pi_Ru)$ is positive for sufficiently large $R$.

\begin{lemma}
\label{th:disl_trunc_stab}
There exist constants $\ldR$ such that
\begin{equation}
  \<\ddel E^\L\b(\yh+\Pi_Ru\b)v,v\>\geq \ldR \|Dv\|_2^2\qquad
    \text{for all}\quad v\in\Usz,
\end{equation}
and $\ldR\to\ld > 0$ as $R\to\infty$. 
\end{lemma}

\begin{proof}
  Noting that $\|Dv\|_\infty\leq\|Dv\|_2$ for any $v\in\Hsi(\L)$,
\begin{align*}
  \<\del^2 E^\L(\yh+\Pi_Ru)v,v\> &= \b\< [\del^2 E^\L(\yh+\Pi_Ru)
    -\del^2 E^\L(\yh+u)]v,v\b\>+\b\<\del^2 E^\L(\yh+u)v,v\b\>,\\
  &\geq \b(\ld-\|\psi'''\|_\infty \b\|D\Pi_Ru-Du\b\|_\infty\b) \|Dv\|_2^2,\\
  &\geq \b(\ld-\eps_R\b) \|Dv\|_2^2,
\end{align*}
where $\eps_R\lesssim R^{-1}$ as $R\to\infty$ by Lemma \ref{th:trunc_est}.
\end{proof}

\subsection{Inverse Function Theorem}
We review a quantitative version of the inverse function theorem, adapted from
\cite[Lemma B.1]{LuskinOrtner13}.

\begin{lemma}
\label{th:inv_func_theorem}
Let $X,Y$ be Hilbert spaces, $w\in X$, $F\in C^2(B_R^X(w);Y)$ with Lipschitz
continuous Hessian, $\|\ddel F(x)-\ddel F(y)\|_{L(X,Y)}\leq M\|x-y\|_X$ for
any $x,y\in B_R^X(w)$. Furthermore, suppose that there exist constants
$\mu,r>0$ such that
\begin{equation*}
  \<\ddel F(w)v,v\>\geq \mu\|v\|_X^2,\quad\|\del F(w)\|_Y\leq r,\quad
    \text{and}\quad\smfrac{2Mr}{\mu^2}<1,
\end{equation*}
then there exists a locally unique $\bar{w}\in B_R^X(w)$ such that
$\del F(\bar{w})=0$, $\|w-\bar{w}\|_X\leq \frac{2r}{\mu}$ and 
\begin{equation*}
  \<\ddel F(\bar{w})v,v\>\geq \b(1-\smfrac{2Mr}{\mu^2}\b)\mu\|v\|_X^2.
\end{equation*}
\end{lemma}

\section{Proof for the Infinite Lattice}
\label{sec:inf}
Before considering the case of finite lattice domains, we first set
out to prove Theorem \ref{th:full_latt} in the case when $\Om=\L$. In
this case we are able to give a substantially simplified argument that
concerns only the interaction between dislocations rather than the
additional difficulty of the interaction of dislocations with the
boundary which is present in the finite domain case.

\subsection{Analysis of the predictor}
\label{sec:inf:pred}
Suppose that $\D$ is a dislocation configuration in $\L$: we define an
approximate solution (predictor) with truncation radius $R$ to be
\begin{equation}
  z := \sum_{(C,s)\in\D} s\,\b(\yh + \Pi_R u\b) \circ G^C.  
  \label{eq:appr_sol_full}
\end{equation}
The following lemma provides an estimate on the residual of such approximate
solutions in terms of $L_\D$.

\begin{lemma}
  \label{th:full_latt_cons}
  Suppose $z$ is the approximate solution for a dislocation
  configuration $\D$ in $\L$ as defined in (\ref{eq:appr_sol_full})
  with truncation radius $R = L_\D/5$. Then there exists $L_0 =
  L_0(N)$ and a constant $c = c(N)$, such that, whenever $L_\D > L_0$,
  \begin{displaymath}
    \b\| \del E(z) \b\|_{\Hsi(\L)^*} \leq c L_\D^{-1}.
  \end{displaymath}
\end{lemma}

\begin{proof}
  Enumerate the elements of $\D$ as $(C^i,s^i)$ where $i=1,\ldots,N$.
  Setting $G^i:=G^{C^i}$, let $y^i := (\yh + \Pi_R u) \circ G^i$ and
  $\yh^i := \yh \circ G^i$.
  Let $r := 2 (R+1) = 2 (L_\D/5+1)$ and $v$ be any test function in
  $\Hsi(\L)$. Define
  \begin{equation*}
    v^i:=\Pi^{C^i}_{r} \! v\quad\text{for }i=1,\ldots,N,\qquad\text{and} \qquad v^0 := v -
    \sum_{i = 1}^N v^i.
  \end{equation*}
  Lemma \ref{th:trunc_est} implies that $\|Dv^i\|_{2} \lesssim \| Dv
  \|_2$ for $i = 0, \dots, N$.

  Assumption {\bf(STAB)} implies $\del E(\yh+u) = 0$, so we may decompose the
  residual into
  \begin{align}
    \notag
    \<\del E(z),v\> &= \sum_{i = 0}^N \<\del E(z),v^i\>,\\
    \notag
    &= \b\<\del E(z),v^0 \b\>+\sum_{i\neq0}\b\<\del E(z)
    -\del E(y^i),v^i \b\>\\
    \notag
    &\qquad\qquad+\sum_{i\neq0}\b\<\del E(y^i)
    -\del E\b(\yh^i+u\circ G^i\b),v^i \b\> \\
    &=: {\rm T}_1 + {\rm T}_2 + {\rm T}_3.
    \label{eq:inf:pred_prf:decomp}
  \end{align}

  {\it The term ${\rm T}_1$: } Employing Lemma
  \ref{th:residual_sum_yhat}, and using the fact that $z = \sum_{i =
    1}^N \yh \circ G^i$ in ${\rm supp}(v^0)$ we obtain that
  \begin{align}
    \notag
    \b|{\rm T}_1\b| &= \b|\< \del E(z), v^0 \>\b| \leq \sum_{b \in
      \Bonds} |g_b| |Dv^0_b|
    \lesssim \sum_{b \in \Bonds} \sum_{i = 1}^N \dist(b, C^i)^{-3}
    |Dv^0_b| \\
    &\lesssim \sum_{i = 1}^N \bg( \sum_{\substack{b \in \Bonds \\
        \dist(b, C^i) \geq r/2-1}}\dist(b, C^i)^{-6} \bg)^{1/2}
    \|Dv^0\|_2  \lesssim r^{-2} \| D v\|_2.
    \label{eq:inf:predprf:T1}
  \end{align}

  {\it The term ${\rm T}_2$: } Here, we have $z- y^i = \sum_{j \neq i}
    \yh^j$ in the support of $v^i$. We expand
  \begin{align}
    \b\<\del E(z)-\del E\b(y^i\b),v^i\b\> &=
    \sum_{b \in \Bonds} \psi''(s_b) \sum_{j \neq i} D \yh^j_b Dv^i_b \notag
      \\
    &= \psi''(0) \sum_{j \neq i} \sum_{b \in \Bonds} D \yh^j_b Dv^i_b
    + \sum_{b \in \Bonds} h_b Dv^i_b,\label{eq:inf:predprf:T2split}
  \end{align}
  where $|s_b| \lesssim (1+\dist(b, C^i))^{-1}$ and 
  \begin{displaymath}
    |h_b| = \B|\b(\psi''(s_b) - \psi''(0)\b) \sum_{j \neq i} D
    \yh^j_b\B| \lesssim (1+\dist(b, C^i))^{-2}  L_\D^{-1}.
  \end{displaymath}
  We have Taylor expanded and used the evenness of $\psi$ to arrive at
  the estimate on the right.
  The first group of terms in \eqref{eq:inf:predprf:T2split} can be
  estimated as in (\ref{eq:inf:predprf:T1}) to obtain $|\sum_{b \in \Bonds}
  D \yh^j_b Dv^i_b| \lesssim L_\D^{-2}\|Dv\|_2 $ for all $j \neq i$. For the 
  second group in \eqref{eq:inf:predprf:T2split}, we have
  \begin{displaymath}
    \B|\sum_{b \in \Bonds} h_b Dv^i_b\B| \lesssim L_\D^{-1}
      \bg(\sum_{\substack{b\in \Bonds \\ \dist(b, C^i) \leq r+1}}
    (1+\dist(b, C^i))^{-4} \bg)^{1/2} \| D v^i \|_2
    \lesssim L_\D^{-1}\| Dv \|_2.
  \end{displaymath}
 
  {\it The term ${\rm T}_3$: } The final group in
  (\ref{eq:inf:pred_prf:decomp}) is straightforward to estimate using
  the truncation result of Lemma \ref{th:trunc_est}, giving
  \begin{equation*}
    \b|\b\<\del E(y^i)
    -\del E(\yh^i+u\circ G^i),v^i\b\>\b| \leq \|\psi''\|_\infty \|D\Pi_Ru - Du\|_2 \|Dv^i\|_2 \lesssim R^{-1}
    \| Dv \|_2.
  \end{equation*}

  {\it Conclusion: } Inserting the estimates for ${\rm T}_1, {\rm
    T}_2, {\rm T}_3$ into (\ref{eq:inf:pred_prf:decomp}) we obtain
  \begin{equation*}
    \b|\<\del E(z), v\>\b| \lesssim \B(r^{-2} + L_\D^{-1} + R^{-1}\B)\|Dv\|_2 \lesssim L_\D^{-1} \|Dv\|_2.
    \qedhere
  \end{equation*}
\end{proof}

\subsection{Stability of the predictor}
\label{sec:inf:stab}
We proceed to prove that $\ddel E(y)$ is positive, where $y$ is the
predictor constructed in \eqref{eq:appr_sol_full}. This result employs
ideas similar to those used in the proof of \cite[Theorem 4.8]{EOS13},
modified here to an aperiodic setting and extended to cover the case of
multiple defect cores.

\begin{lemma}
\label{th:full_latt_stab}
Let $z$ be a predictor for a dislocation configuration $\D$ in $\L$, as
defined in \eqref{eq:appr_sol_full}, where $|\D|=N$. Then there exist
postive constants $R_0 = R_0(N)$ and $L_0 = L_0(N)$ such that if
$R\geq R_0$ and $L_\D\geq L_0$, there exists $\lLR \geq \lambda_d / 2$
so that
\begin{equation*}
  \<\del^2 E(z)v,v\>\geq \lLR\|Dv\|_2^2 \qquad \text{for all } v \in \Hsi(\L).
\end{equation*}
\end{lemma}

\begin{proof}
Lemma \ref{th:disl_trunc_stab} implies the existence of $R_0$ such that
$\ldR\geq3\ld/4>0$ for all $R\geq R_0$, thus we choose a truncation radius
$R\geq R_0$ which will remain fixed for the rest of the proof.
We now argue by contradiction. Suppose that there exists no $L_0$
satisfying the statement; it follows that there exists $\D^n$, a sequence of
dislocation configurations such that
\begin{enumerate}
  \item $N:=|\D^n|$, $\b|\b\{(C,+1)\in\D^n\b\}\b|$ and 
    $\b|\b\{(C,-1)\in\D^n\b\}\b|$ are constant,
  \item $L^n:=L_{\D^n}\to\infty$ as $n\to\infty$ and
  \item $\ddel E(z^n)<\ld/2$ for all $n$, where $z^n$ is the approximate
    solution corresponding to the configuration $\D^n$ in $\L$ with
    truncation radius $R$, as defined in \eqref{eq:appr_sol_full}.
\end{enumerate}
The first condition may be assumed without loss of generality by taking 
subsequences. We enumerate the elements $(C^{n,i},s^{n,i})$ of $\D^n$, and
write $G^{n,i}:=G^{C^{n,i}}$ and $H^{n,i}:=H^{C^{n,i}}$. By translation
invariance and the fact that $\psi$ is even, we may assume without further
loss of generality that $(C^{n,1},s^{n,1})=(C_0,+1)$. For each $n$,
\begin{equation*}
  \lambda_n:=\inf_{\substack{v\in\Hsi(\L)\\\|Dv\|_2 =1}}
    \<\ddel E(z^n)v,v\><\ld/2
\end{equation*}
exists since, for any $z\in\Us(\L)$, $\ddel E(z)$ is a bounded bilinear
form on $\Hsi(\L)$. Let $v^n\in\Usz(\L)$ be a sequence of test functions
such that $\|Dv^n\|_2=1$ and
\begin{equation*}
  \lambda_n\leq \<\ddel E(z^n)v^n,v^n\>\leq \lambda_n + n^{-1}.
\end{equation*}
Since $v^n$ is bounded in $\Hsi(\L)$, it has a weakly convergent subsequence.
By the translation invariance of the norm and taking further subsequences
without relabelling, we further assume that
$\bar{v}^{n,i}:=v^n\circ H^{n,i}$ weakly converges for each $i$.
We now employ the result of \cite[Lemma 4.9]{EOS13}. This states that there 
exists a sequence of radii, $r^n\to\infty$, for which we may also assume
$r^n\leq L^n/3$, so that for each $i=1,\ldots,N$,
\begin{equation*}
  w^{n,i}:=\Pi^{C^{n,i}}_{r^n}v^n\quad\text{satisfies}\quad
    w^{n,i}\circ H^{n,i}\to\bar{w}^i\quad\text{and}\quad
    (v^n-w^{n,i})\circ H^{n,i}\wto0\quad\text{in }\Hsi(\L).
\end{equation*}

Writing $\bar{w}^{n,i} := w^{n,i}\circ H^{n,i}$, and defining
$w^{n,0}:=v^n-\sum_{i=1}^Nw^{n,i}$, it follows that
\begin{equation*}
  \<\ddel E(z^n)v^n,v^n\> = \sum_{i,j=0}^N\<\ddel E(z^n)w^{n,i},w^{n,j}\>
   =\sum_{i=0}^N\<\ddel E(z^n)w^{n,i},w^{n,i}\>
    +2\sum_{i=1}^N\<\ddel E(z^n)w^{n,0},w^{n,i}\>,
\end{equation*}
where, by choosing $r^n\leq L^n/3$, we have ensured that $\supp\{w^{n,i}\}$ 
for $i=1,\ldots,N$ only overlaps with $\supp\{w^{n,0}\}$, and hence all
other `cross--terms' vanish. For $i=1,\ldots,N$,
\begin{align}
  \<\ddel E(z^n)w^{n,i},w^{n,i}\>&=\<[\ddel E(z^n\circ H^{n,i})
    -\ddel E(\yh+\Pi_Ru)]\bar{w}^{n,i},\bar{w}^{n,i}\>
    +\<\ddel E(\yh+\Pi_Ru)\bar{w}^{n,i},\bar{w}^{n,i}\>,\notag\\
  &\geq  \B(\ldR-\smfrac{N\,\|\psi'''\|_\infty}{2L^n/3}\B)\|Dw^{n,i}\|_2^2.
    \label{eq:full_latt_stab_1}
\end{align}
For the $i=0$ term, we have
\begin{align}
  \<\ddel E(z^n)w^{n,0},w^{n,0}\>&=\<[\ddel E(z^n)
    -\ddel E(0)]w^{n,0},w^{n,0}\>
    +\<\ddel E(0)w^{n,0},w^{n,0}\>,\notag\\
  &\geq \B(\psi''(0)-\smfrac{N\|\psi'''\|_\infty}{r^n}\B)\|Dw^{n,0}\|_2^2.
    \label{eq:full_latt_stab_2}
\end{align}
For the cross--terms, Since we assumed that $r^n\leq L^n/3$,
we deduce that
\begin{equation*}
  \<\ddel E(z^n)w^{n,0},w^{n,i}\>=\<\ddel E(z^n)(v^n-w^{n,i}),w^{n,i}\>.
\end{equation*}
Using the translation invariance of $E$, and adding and subtracting terms,
we therefore write
\begin{align*}
  \b\<\ddel E(z^n)w^{n,0},w^{n,i}\b\>&=\b\<\b[\ddel E(z^n\circ H^{n,i})
    -\ddel E(\yh+\Pi_Ru)\b](\bar{v}^{n,i}-\bar{w}^{n,i}),\bar{w}^{n,i}\b\>\\
  &\qquad+\b\<\ddel E(\yh+\Pi_Ru)(\bar{v}^{n,i}-\bar{w}^{n,i}),
    \bar{w}^{n,i}-\bar{w}^i\b\>\\
  &\qquad\qquad+\b\<\ddel E(\yh+\Pi_Ru)(\bar{v}^{n,i}-\bar{w}^{n,i}),
   \bar{w}^i\b\>,\\
  &=:{\rm T}_1+{\rm T}_2+{\rm T}_3.
\end{align*}
Estimating the first two terms on the right hand side, we obtain:
\begin{gather*}
  {\rm T}_1\leq \smfrac{N\,\|\psi'''\|_\infty}{2L^n/3}\B(\|Dv^{n,i}\|_2
   +\|Dw^{n,i}\|_2\B)\|Dw^{n,i}\|_2 \leq\smfrac{N\,\|\psi'''\|_\infty}{L^n/3}
   \qquad\text{and}\\
  {\rm T}_2\leq\|\psi''\|_\infty\B(\|Dv^{n,i}\|_2
   +\|Dw^{n,i}\|_2\B)\|D\bar{w}^{n,i}-D\bar{w}^i \|_2,
\end{gather*}
both of which converge to 0 as $n\to\infty$.
Since $\bar{v}^{n,i}-\bar{w}^{n,i}\wto0$ as $n\to\infty$, it follows that
${\rm T}_3\to0$ as well, and hence
\begin{equation}
  \<\ddel E(z^n)w^{n,0},w^{n,i}\>\to0\label{eq:lb_xterms}
\end{equation}
as $n\to\infty$ for each $i$. Putting \eqref{eq:lb_xterms} and the result of
Lemma \ref{th:0_glob_stab} together with \eqref{eq:full_latt_stab_1} and
\eqref{eq:full_latt_stab_2},
\begin{equation}
  \<\ddel E(z^n)v^n,v^n\>\geq (\ldR-\eps^n)\sum_i\|Dw^{n,i}\|_2^2 + \eps^n,
\end{equation}
where $\eps^n\to0$ as $n\to\infty$. All that remains is to verify that
\begin{equation}
  \liminf_{n\to\infty} \B(\sum_i\|Dw^{n,i}\|_2^2-\|Dv^n\|_2^2\B)\geq 0.
    \label{eq:liminf_norms}
\end{equation}
By definition, $\sum_i |Dw^{n,i}_b|^2 \neq |Dv^n_b|^2$ only when $b\in
\supp\{Dw^{n,i}\}\cap\supp\{Dw^{n,0}\}$ for some $i=1,\ldots,N$.
In such cases,
\begin{equation*}
  |Dw^{n,0}_b|^2+|Dw^{n,i}_b|^2-|Dv^n_b|^2 = -2\,Dw^{n,0}_bDw^{n,i}_b.
\end{equation*}
Therefore, consider
\begin{align*}
  \delta^{n,i}&:=\sum_{b\in\Bonds}Dw^{n,0}_bDw^{n,i}_b
    =\sum_{b\in\Bonds}\b(Dw^{n,0}\circ H^{n,i}\b)_b
      \b(D\bar{w}^{n,i}_b-D\bar{w}^i_b\b)
    +\sum_{b\in\Bonds}\b(Dw^{n,0}\circ H^{n,i}\b)_bD\bar{w}^i_b.
\end{align*}
Since $w^{n,0}\circ H^{n,i}\wto0$ and $\bar{w}^{n,i}\to \bar{w}^i$, it follows
that $\delta^{n,i}\to0$, and thus \eqref{eq:liminf_norms} holds. Further,
\begin{equation*}
  \lambda_n + n^{-1} \geq \<\ddel E(z^n)v^n,v^n\>
    \geq(\ldR-\eps^n)\B(1-\sum_i\delta^{n,i}\B)+\eps^n,
\end{equation*}
and so for $n$ sufficiently large, it is clear that
$\lambda_n\geq2\ldR/3\geq\ld/2>0$, which contradicts the
assumption that $\lambda_n<\ld/2$ for all $n$.
\end{proof}

\subsection{Conclusion of the proof of Theorem \ref{th:full_latt},
  Case \texorpdfstring{$\Om = \L$}{Omega=Lambda}}
\subsubsection{Proof of (2)}
\label{sec:inf:conclusion2}
Lemma \ref{th:full_latt_cons} and Lemma \ref{th:full_latt_stab} now enable 
us to state that there exist $L_0$ and $R_0$ depending only on $N=|\D|$ such
that whenever $\D$ satisfies $L_\D\geq L_0$, $R\geq R_0$, and $z$ is an approximate solution corresponding to $\D$ with truncation radius $R$,
\begin{equation*}
  \lLR\geq \mu:=\frac\ld2>0,\qquad\text{and}\qquad
    \|\del E(z)\|< r:=\min\bg\{\frac{c\ld}{4L_\D},
    \frac{\ld^2}{16\,\|\psi'''\|_\infty}\bg\}.
\end{equation*}
We note that
\begin{equation*}
  \|\ddel E(z+u)-\ddel E(z+v)\|\leq \|\psi'''\|_\infty \|Du-Dv\|_2,
\end{equation*}
so setting $M := \|\psi'''\|_\infty$,
we may apply Lemma \ref{th:inv_func_theorem}, since
$\frac{2Mr}{\mu^2}\leq\smfrac12<1$.
It follows that there exists
$w\in\Hsi(\L)$ with $\|Dw\|_2\leq c'L_\D^{-1}$ such that
\begin{equation*}
  \del E(z+w)=0,\qquad
    \<\ddel E(z+w)v,v\>\geq\frac\ld4\|Dv\|_2^2,
\end{equation*}
and so $z+w$ is a strongly stable equilibrium. The constant $c'$ depends
only on $N$, the number of dislocation cores, establishing item (2) of
Theorem \ref{th:full_latt}.

\subsubsection{Proof of (1)}
\label{sec:inf:conclusion1}
We begin by increasing $R_0$ if necessary to ensure that
$\frac{N}{2\pi R_0}\leq\frac14$. Suppose that $z$ is a predictor for a 
configuration $\D$ in $\L$ satisfying $L_\D\geq L_0$ and $R\geq R_0$.
If $\alpha\in[Dz]$, by increasing $R_0$, we have ensured that
\begin{equation*}
  \alpha_b\in[-\smfrac14,\smfrac14]\qquad\text{for any }b\notin
    \bigcup_{(C,s)\in\D}\supp\{D\Pi_R^Cu\},
\end{equation*}
and furthermore
\begin{equation*}
 \alpha_b = \sum_{(C,s)\in\D} s(\alh\circ G^C)_b \qquad
    \text{for any }b\notin\bigcup_{(C,s)\in\D}\supp\{D\Pi_R^Cu\}.
\end{equation*}
Let $\alpha'\in[Dz+Dw]$, and so if $L_\D > 4c'$, where $c'$ is the constant
arising in the proof of (2), $z+w$ is a strongly stable local
equilibrium such that $\|Dw\|_\infty\leq\|Dw\|_2<\smfrac14$.
When $b\notin\supp\{D\Pi_R^{C}u\}$ for any $(C,s)\in\D$, this choice
entails that $\alpha'_b = \alpha_b+Dw_b$.

Taking $A$ to be a collection of positively-oriented cells such that $B_R(0)
\subset \mathrm{clos}(A)\subset B_{L_\D/2}(0)$ and 
setting $A^C:=H^C(A)$,
\begin{gather*}
  \int_{\partial A^{C'}}\alpha'
    = \int_{\partial A^{C'}}\sum_{(C,s)\in\D}s\,(\alh\circ G^C) +Dw
    = s'\qquad\text{for any }(C',s')\in\D,\qquad\text{and}\\
  \int_{\partial C} \alpha' = 0\qquad
    \text{for any}\quad C\notin\bigcup_{(C,s)\in\D} A^C,\quad
    \text{implying}\quad\Cores[\alpha']\subset\bigcup_{(C,s)\in\D} H^C(A).
\end{gather*}

\subsubsection{Proof of (3)}
We divide the proof of statement (3) into two cases: $B[z+w]=0$,
and $|B[z+w]|>1$.

Suppose $z+w$ is a strongly stable equilibrium such that $B[z+w]=0$, 
arising from statement (2) of Theorem \ref{th:full_latt}.
It follows that $|\{(C,1)\in\D\}|=|\{(C,-1)\in\D\}|$, so
enumerating pairs $(C^i_+,1),(C^i_-,-1)\in\D$, we define
\begin{equation}
  v^i(x):=\smfrac1{2\pi}\b[\arg\b(x-x^{C^i_+}\b)-\arg\b(x-x^{C^i_-}\b)\b],
    \quad\text{and}\quad v(x):=\sum_iv^i(x),\label{eq:dip_constr}
\end{equation}
where $v^i$ is a function with a branch cut of finite length. As for 
approximate solutions $z$, we may extend $\nabla v^i$ to a function
which is $\CC^\infty(\R^2\setminus\{x^{C^i_+},x^{C^i_-}\};\R^2)$.
It may then be directly verified that $|\nabla v^i(x)| \lesssim |x|^{-2}$
for $|x|$ suitably large, and hence when $v^i$ is understood as a function
in $\Us(\L)$, it follows that $v^i\in\Hsi(\Om)$.

It may now be checked that $Dz_b-Dv_b\in\Z$ for all $b\in\Bonds$, and hence
\begin{equation*}
  E(z-v;z+w) = E(0;z+w) =-E(z+w;0) <0,
\end{equation*}
as Lemma \ref{th:0_glob_stab} implies that $0=\argmin_{u\in\Hsi(\Om)} E(u;0)$,
which contradicts the assumption that $z+w$ was a globally stable equilibrium.
\medskip

If $|B[z+w]|>1$, then without loss of generality, we suppose $B[z+w]>1$. We
will only consider the case where $B[z+w]=2$ here, leaving the general case
for the interested reader. Suppose for contradiction that $z+w$ is
a strongly stable equilibrium given by (2) in Theorem \ref{th:full_latt}
with $B[z+w]=2$, and that $z+w$ is additionally globally stable. If true, then any configuration of the form
\begin{equation}
  y = \yh+\yh\circ G^C \label{eq:inf:counterex_func}
\end{equation}
must satisfy $E(y;z+w)\geq0$, since by a similar argument to that used in the
previous case, we may define $y$ such that $y-z\in\Hsi(\L)$.
Our strategy is to construct a sequence $y^n$ of the form 
\eqref{eq:inf:counterex_func} such that $E(y^{n+1};y^n)\leq -C<0$,
and hence prove a contradiction. To that end, define a sequence of
cells $C^n$ such that $x^{C^n}$ lies on the positive $x$-axis for all $n$, with
\begin{equation*}
   \dist(C^n,C^{n+1})<\dist(0,C^n),\quad\dist(0,C^0)\geq K\quad\text{and}
     \quad\dist(C^{n-1},C^n)\geq K,
\end{equation*}
where $K$ is a parameter we will choose later. Our choice of $y^n$ is then
\begin{equation*}
  y^n:=\yh+\yh\circ G^{C^n}.
\end{equation*}
If $K$ is sufficiently large, we note that $\alpha^n\in[Dy^n]$ is unique. 
Letting $v^n:=y^n-y^{n-1}$, decompose $Dv^n=\beta^n+Z^n_b$,
where $\beta^n = \alpha^n_b-\alpha^{n-1}_b$, and $Z^n_b=Dv^n_b-\beta^n_b$
has support only on bonds crossing the $x$-axis between $x^{C^{n-1}}$ and
$x^{C^n}$.

We consider $E(y^{n-1};y^n)=-E(y^n;y^{n-1})$. It may be checked directly
that $\beta^n\in\ell^2(\Bonds)$ for any $n$, and is uniformly bounded
in $\ell^p(\Bonds)$ for any $p>2$. Using these facts, the decay of $\alh_b$,
and Taylor expanding, we obtain that for any $\eps>0$, there exists a
constant $C$ depending only on $\psi$ and $\eps$ such that
\begin{align*}
  E(y^{n-1};y^n) &= \sum_{b\in\Bonds}\psi(\alpha^{n-1}_b)-\psi(\alpha^n_b)\\
    &= \sum_{b\in\Bonds}\psi(\alpha^n_b-\beta^n_b)-\psi(\alpha^n_b)
    -\psi'(\alpha^n_b)(-\beta^n_b)+\<\del E(y^n),-Dv^n\>
    -\sum_{b\in\Bonds}\psi'(\alpha^n_b)(-Z^n_b),\\
  &\geq \smfrac12(\psi''(0)-\eps)\|\beta^n\|_2^2-C
    -\<\del E(y^n),Dv^n\>+\sum_{b\in\Bonds}\psi'(\alpha^n_b)Z^n_b.
\end{align*}
Employing the result of Lemma \ref{th:residual_sum_yhat}, we find that we may
write
\begin{equation*}
  -\<\del E(y^n),Dv^n\>+\sum_{b\in\Bonds}\psi'(\alpha^n_b)Z^n_b
    = \sum_{b\in\Bonds}-g^n_b \beta^n_b + h^n_b Z^n_b\geq -\|g^n\|_2\|\beta^n\|_2
      +\sum_{b\in\Bonds} h^n_bZ^n_b.
\end{equation*}
It may be verified that $\|g^n\|_2$ is uniformly bounded in $n$, using the 
properties demonstrated in Lemma \ref{th:residual_sum_yhat}, and that
$Z^n_b$ is negative on bonds of the form $(\xi,\xi+a_2)$
or $(\xi,\xi+a_3)$ crossing the $x$-axis. Since by assumption
$\dist(C^{n-1},C^n)<\dist(0,C^{n-1})$, $h^n_b$ is negative for all bonds
in $\supp\{Z^n\}$ --- in particular, these assertions imply that
\begin{equation*}
  \sum_{b\in\Bonds}h^n_bZ^n_b\geq0,\quad\text{and so}\quad
  E(y^{n-1};y^n)\geq c_0 \|\beta^n\|_2^2-c_1
\end{equation*}
for some constants $c_0,c_1>0$ which depend only on $\psi$.
Applying Jensen's inequality to $\beta^n$ on a series of closed curves around
$C^n$, we find that
\begin{equation*}
  \|\beta^n\|_2^2\geq c\,\log(\dist(C^{n-1},C^n))\geq c'\log(K),
\end{equation*}
where $c,c'$ are constants depending only on the lattice,
and hence as long as $K$ is suitably large, we have that $E(y^{n-1};y^n)\geq
C\geq0$. Thus
\begin{equation*}
  E(y^k;z+w) = \sum_{n=1}^k E(y^n;y^{n-1})+ E(y^0;z+w) \leq -Ck + E(y^0;z+w),
\end{equation*}
and letting $k\to\infty$, we have a contradiction to the fact that $z+w$ is a
globally stable equilibrium.

\section{Proof for Finite Lattice Polygons}
\label{sec:fin}
As in \S\ref{sec:inf}, we construct approximate solutions and prove
estimates on the derivative and Hessian of the energy evaluated at
these points so that we may apply Lemma \ref{th:inv_func_theorem}.
The two main differences between this and the preceding analysis are
(i) $z$ defined in \ref{eq:appr_sol_full} does {\em not} satisfy the
natural boundary conditions of Laplace's equation in a finite domain,
and (ii) we must estimate residual force contributions at the boundary,
which cannot be achieved by a simple truncation argument as used in
\S\ref{th:trunc_est} --- at this stage the fact that $\Om$ has a
boundary plays a crucial role.

To obtain a predictor satisfying the natural boundary conditions we
introduce a {\em boundary corrector}, $\ycorr \in C^1(U^\Omega) \cap
C^2({\rm int}(U^\Omega))$, corresponding to a configuration
$\D$ in $\Om$ which satisfies
\begin{equation}
  -\Delta \ycorr = 0\quad\text{in }U^\Om,\qquad\nabla \ycorr\cdot\nu=
    -\sum_{(C,s)\in\D}s\nabla(\yh\circ G^C)\cdot \nu
    \quad\text{on }\partial U^\Om,
\label{eq:y_corr_prob}
\end{equation}
where $\nu$ is the outward unit normal on $\partial U^\Om$. 
\S\ref{sec:fin:corrector} is devoted to a study of this problem and its
solution.

We then define an approximate solution (predictor) corresponding
to $\D$ in $\Om$ with truncation radius $R$ as
\begin{equation}
  z := \sum_{(C,s)\in\D}s\,\b(\yh + \Pi_R u\b) \circ G^C
    +\ycorr.
    \label{eq:appr_sol_fin}
\end{equation}

\subsection{The continuum boundary corrector}
\label{sec:fin:corrector}
Here, as remarked above, we give proofs of several important facts about the
boundary corrector. Since we are considering a boundary value problem
in a polygonal domain, we use the theory developed in \cite{Grisvard}
to obtain regularity of solutions to \eqref{eq:y_corr_prob}.

Noting that the boundary corrector problem is linear, it suffices to
analyse the problem when only one positive dislocation is present at a point
$x'\in U^\Om$. We therefore consider the problem
\begin{equation}
  -\Delta \ycorr = 0\quad\text{in }U^\Om,\qquad\nabla \ycorr\cdot\nu=
    g_m\quad\text{on }\Gamma_m,
\label{eq:y_corr_prob_1disl}
\end{equation}
where as in \S\ref{sec:conv_crystals}, $\Gamma_m$ are the straight segments of
$\partial U^\Om$ between corners $(\kappa_{m-1},\kappa_m)$, $\nu$ is the
outward unit normal, and
\begin{equation*}
  g_m(s):=-\nabla\yh(s-x')\cdot\nu\qquad\text{for }s\in\Gamma_m.
    \label{eq:ycorr_bcs}
\end{equation*}
As before, by $\nabla\yh(x-x')$ we mean the extension of the gradient of
$\yh(x-x')$ to a function in $\CC^\infty\b(\R^2\setminus\{x'\}\b)$.
Since $\nu$ is constant along $\Gamma_m$, it follows that
$g_m\in\CC^\infty(\Gamma_m)$, and so applying Corollary 4.4.3.8 in
\cite{Grisvard}, it may be seen that this problem has a solution in
$\HH^2(U^\Om)$ which is unique up to an additive constant, as long
as $\int_{\partial U^\Om}g=0$. This condition may be verified
by standard contour integration techniques, for example. Furthermore,
$\ycorr$ is harmonic in the interior of $U^\Om$, and hence analytic
on the same set.

We now obtain several bounds for solutions of the problem
\eqref{eq:y_corr_prob_1disl} in terms of $\dist(x',\partial U^\Om)$, taking 
note of the domain dependence of any constants. 
The key fact used to construct these estimates is that
$\yh+\ycorr$ is a harmonic conjugate of the Green's function for the Laplacian
with Dirichlet boundary conditions on $U^\Om$.

\begin{lemma}
Suppose $U^\Om$ is a convex lattice polygon, and $\ycorr$ solves
\eqref{eq:y_corr_prob_1disl}. Then there exist constants $c_1$ and $c_2$
which are independent of the domain such that
\begin{gather}
  |\nabla \ycorr(x)|\leq c_1\,\dist(x,x')^{-1}\quad\text{for any }x\in U^\Om,
    \qquad\|\nabla\ycorr\|_\infty\leq c_1\,\dist(x',\partial U^\Om)^{-1},
    \label{eq:Dycorr_bnd}\\
  \text{and}\quad\|\nabla^2\ycorr\|_{\LL^2(U^\Om)}\leq c_3 \frac{\log(
   \dist(x',\partial U^\Om))}{\dist(x',\partial U^\Om)}.
   \label{eq:D2ycorr_bnd}
\end{gather}
\end{lemma}

\begin{proof}
We begin by noting that $\yh(x-x')=\smfrac{1}{2\pi}\arg(x-x')$ is a harmonic
conjugate of $\smfrac{1} {2\pi}\log(|x-x'|)$, and we will further demonstrate
that $\ycorr$ is a harmonic conjugate of $\Psi$, the solution of the
Dirichlet boundary value problem
\begin{equation*}
  -\Delta \Psi(x)=0\quad\text{in }U^\Om,\qquad \Psi(s) = -\smfrac{1}{2\pi}
    \log(|x-x'|)\quad\text{on }\partial U^\Om.
\end{equation*}
By virtue of Corollary 4.4.3.8 in \cite{Grisvard}, there exists
a unique $\Psi\in\HH^2(U^\Om)$ solving this problem, and
since $\Psi$ is harmonic in $U^\Om$, a simply connected region,
a harmonic conjugate $\Psi^*$ exists. By definition, $\Psi^*$ satisfies
the Cauchy--Riemann equations
\begin{equation}
  \nabla \Psi^*(x) = \mR_4^T\nabla \Psi(x)\quad\text{for all }x\in U^\Om,
    \label{eq:CR_eqns}
\end{equation}
where $\mR_4$ is the matrix corresponding a positive rotation through
$\smfrac\pi 2$ about the origin. In particular,
\begin{equation*}
  \frac{\partial \Psi^*}{\partial \nu} = \frac{\partial \Psi}{\partial \tau}=
    \frac{(x-x')}{2\pi|x-x'|^2}\cdot\mR_4\nu=-\nabla\yh(x-x')\cdot\nu
    \quad\text{on }\partial U^\Om,\quad \text{and} \quad -\Delta \Psi^*=0
    \text{ in }U^\Om,
\end{equation*}
where $\tau$ is the unit tangent vector to $\partial U^\Om$ with the positive
orientation. By uniqueness of solutions for \eqref{eq:y_corr_prob_1disl},
it follows that $\Psi^*=\ycorr$ up to an additive constant, and hence
$\ycorr$ is a harmonic conjugate of $\Psi$. Furthermore, by differentiating
\eqref{eq:CR_eqns},
\begin{equation}
  \|\nabla^2\Psi\|_{\LL^2(U^\Om)}=\|\nabla^2\ycorr\|_{\LL^2(\Om)}.
    \label{eq:fin:D2y=D2Psi}
\end{equation}
The identities \eqref{eq:CR_eqns} and \eqref{eq:fin:D2y=D2Psi} will allow
us to use estimates on the derivatives of $\Psi$ to directly deduce
\eqref{eq:Dycorr_bnd} and \eqref{eq:D2ycorr_bnd}.

To prove \eqref{eq:Dycorr_bnd}, we rely upon Proposition 1 in \cite{Fromm93},
which states that there exists a constant $c_1$ depending only on
$\diam(U^\Om)$ such that
\begin{equation*}
  |\nabla \Psi(x)|\leq c_1\,\dist(x',x)^{-1}.
\end{equation*}
However, as $U^\Om\subset\R^2$, it is straightforward to see by a
change of variables and a scaling argument that the constant $c_1$
cannot depend on $\diam(U^\Om)$, and is therefore independent of the domain
(as long as it remains convex). Taking the Euclidean norm of both sides
in \eqref{eq:CR_eqns} now implies the pointwise bound in
\eqref{eq:Dycorr_bnd}, and the $\LL^\infty$ bound follows immediately
as the partial derivatives of $\ycorr$ satisfy the strong maximum
principle.

To prove \eqref{eq:D2ycorr_bnd}, we use the classical \emph{a priori}
bounds for the Poisson problem. In order to do so, we must introduce an
auxiliary problem with homogeneous boundary conditions.
We therefore seek a solution to
\begin{equation*}
  -\Delta(\Psi-\Phi)=\Delta \Phi\quad\text{in }U^\Om,\qquad \Psi-\Phi\in\HH^2(U^\Om)\cap\HH^1_0(U^\Om),
\end{equation*}
where the function $\Phi:\R^2\to\R$ is defined to be
\begin{align*}
  \Phi(x)&:=-\smfrac{1}{2\pi}\phi\b(|x-x'|/\dist(x',\partial U^\Om)\b)
   \log(|x-x'|),\\\text{where}\quad\phi&\in\CC^\infty([0,+\infty))\quad
   \text{and}\quad\phi(r)=\cases{0 & r\in[0,\smfrac14],\\
    1&r\geq1.}
\end{align*}
By construction $\Phi\in\CC^\infty(U^\Om)$, and $\nabla^2\Phi\in\LL^2(\R^2)$.
Thus $\Delta\Phi\in\LL^2(U^\Om)$, and the existence of a
unique solution $\Psi-\Phi\in\HH^2(U^\Om)\cap\HH^1_0(U^\Om)$ follows from
\cite[Theorem 3.2.1.2]{Grisvard}). Furthermore, inspecting the
proof of \cite[Theorem 4.3.1.4]{Grisvard}, we see that
\begin{equation*}
  \|\nabla^2(\Psi-\Phi)\|_{\LL^2(U^\Om)} = \|\Delta \Phi\|_{\LL^2(U^\Om)}
    \leq \|\nabla^2\Phi\|_{\LL^2(\R^2)},
\end{equation*}
and thus a straightforward integral estimate yields
\begin{equation*}
  \|\nabla^2 \Psi\|_{\LL^2(U^\Om)}\leq\|\nabla^2(\Psi-\Phi)\|_{\LL^2(U^\Om)}+
    \|\nabla^2\Phi\|_{\LL^2(\R^2)}  \leq c_2 \frac{\log\b(\dist(x',\partial U^\Om)\b)}
    {\dist(x',\partial U^\Om)},
\end{equation*}
where $c_2$ is independent of the domain.
\end{proof}

\subsection{Analysis of the predictor}

Here, we prove that the predictor defined in \eqref{eq:appr_sol_fin} is
indeed an approximate equilibrium. Our first step is to formulate the analogue
of Lemma \ref{th:residual_sum_yhat} in the polygonal case.

\begin{lemma}[Finite domain stress lemma]
\label{th:residual_sum_ycorr}
  Let $\Om$ be a convex lattice polygon, $\D$ a dislocation configuration in
  $\Om$ and $z:=\sum_{(C,s)\in\D} \yh\circ G^C+\ycorr$, where $\ycorr$ solves
  \eqref{eq:y_corr_prob}. Then there exist $L_0$ and $S_0$ which depend only
  on $N=|\D|$ such that whenever $L_\D\geq L_0$ and $S_\D\geq S_0$,
  there exist $g:\BOm\to\R$ and $\Sigma:\{b\in\partial W^\Om\}\to\R$
  such that
  \begin{equation*}
    \<\del E^\Om(z),v\> = \sum_{b\in\BOm} g_b Dv_b
      +\sum_{b\in\partial W^\Om}\Sigma_bDv_b,
  \end{equation*}
  and furthermore
  \begin{align}
    |g_b|&\leq c_1\sum_{(C,s)\in\D}\b(1+\dist(b,C)\b)^{-3}
       +c_1\|\nabla^2\ycorr\|_{\LL^2(\omega_b)}&&\text{for all}\quad
       b\notin\partial W^\Om,\label{eq:fin:int_stress_est}\\
    \text{and}\quad|g_b+\Sigma_b|&\leq c_2\sum_{(C,s)\in\D}
      \b(1+\dist(P_\zeta,C)\b)^{-1}+c_1\|\nabla^2\ycorr\|_{\LL^2(\omega_b)}&&\text{for all}\quad b\in
      P_\zeta\subset\partial W^\Om.\label{eq:fin:bdry_stress_est}
  \end{align}
  The constant $c_1$ is independent of the domain, and
  $c_2$ depends linearly on $\ind(\partial W^\Om)$.
\end{lemma}

\begin{proof}
  We begin by choosing $L_0$ and $S_0$ to ensure that $\alpha\in[Dz]$
  is unique: since the constant in \eqref{eq:Dycorr_bnd} is independent
  of the domain, and $D\yh$ has a fixed rate of decay, this choice depends
  only
  on $N$ as stated. Furthermore, we have the representation
  $\alpha_{(\xi,\xi+a_i)} = \int_0^1 \nabla z(\xi+t a_i)\cdot
  a_i\dt$, where $\nabla z$ is to be understood as the
  extension of the gradient of $z$ to a function in $\CC^\infty(U^\Om
  \setminus\bigcup_{(C,s)\in\D}\{x^C\})$.

  Let $\omega_b := \bigcup\{C\in\COm\sep\pm b\in\partial C,
  C\text{ positively oriented}\}$, the union of any cells which $b$
  lies in the boundary of. For $b\notin \partial W^\Om$, $\omega_b$ is
  always a pair of cells, and we set $V:=|\omega_b|$ for any
  $b\notin\partial W^\Om$.
  
  Let $\bar{C}_\epsilon := \bigcup_{(C,s) \in \D} B_\epsilon(x^C)$. If
  $b = (\xi, \xi+a_i)$, define
  \begin{displaymath}
    h_b :=  \frac{\psi''(0)}{V} \lim_{\epsilon \to 0} 
    \int_{\omega_b \setminus \bar{C}_\epsilon}  \nabla z \cdot a_i \dx
    \qquad \text{and} \qquad 
    g_b := \psi'(\alpha_b) - h_b.
  \end{displaymath}
  As in the proof of Lemma \ref{th:residual_sum_yhat}, an application of the
  divergence theorem demonstrates that the former (and hence the latter)
  definition makes sense. Let $v\in\Us(\Om)$, and denote its piecewise
  linear interpolant $Iv$; applying the divergence theorem once more,
  \begin{equation*}
    \sum_{b \in \BOm} h_b Dv_b = 
    \lim_{\epsilon \to 0} \frac{\psi''(0)}{V} \int_{W^\Om \setminus
    \bar{C}_\epsilon} \nabla z \cdot \nabla Iv \dx
    = \frac{\psi''(0)}{V}\int_{\partial W^\Om}Iv\,\nabla z\cdot \nu\ds.
  \end{equation*}
  Recalling the definition of $P_\zeta$ from \eqref{eq:Pzeta_defn}, we
  find that
  \begin{equation*}
    \sum_{b \in \BOm} h_b Dv_b =
      \!\!\sum_{\zeta\in\partial W^\Om\cap\partial U^\Om}
      \frac{\psi''(0)}{V}\int_{P_\zeta} Iv\,\nabla z\cdot\nu\ds.
  \end{equation*}
  By considering the integral over a single period, we may integrate by
  parts
  \begin{equation*}
    \int_{P_\zeta} Iv\,\nabla z\cdot\nu\ds = Iv(\zeta+\tau)
      \int_{P_\zeta} \nabla z\cdot\nu\ds
      -\int_{P_\zeta} Iv'\bg(\int_{\gamma_\zeta^s}\nabla z\cdot\nu\dt\bg)\ds,
  \end{equation*}
  where $\gamma_\zeta^s$ is the arc--length parametrisation of the Lipschitz
  curve following $P_\zeta$ between $\zeta$ and $s$, $\tau$ is the relevant
  lattice tangent vector, and $Iv'$ is the derivative along the curve
  following $P_\zeta$.
  Applying the divergence theorem to the region bounded by $P_\zeta$ and
  $\partial U^\Om$ (as seen on the right of Figure \ref{fig:conv_poly})
  and using the boundary conditions $\nabla z\cdot\nu=0$ on $\partial U^\Om$,
  it follows that $\int_{P_\zeta} \nabla z\cdot\nu=0$. Splitting the domain
  of integration $P_\zeta$ into individual bonds and noting that
  $\nabla Iv$ is constant along each bond,
  \begin{align*}
    \int_{P_\zeta} Iv\,\nabla z\cdot\nu\ds&=
      -\sum_{b\in P_\zeta} \int_{b=(\xi,\xi+a_i)} \hspace{-8mm}\nabla
      Iv\cdot a_i \bg(\int_{\gamma_\zeta^s}\nabla z\cdot \nu\dt\bg) \ds,\\
    &=\sum_{b\in P_\zeta} \Sigma_bDv_b,\quad\text{where}\quad
      \Sigma_b:=-\int_b\int_{\gamma_\zeta^s}\nabla z\cdot\nu\dt\ds.
  \end{align*}
  This concludes the proof of the first part of the statement.

  To obtain \eqref{eq:fin:int_stress_est}, we Taylor expand the potential to
  obtain
  \begin{equation*}
    g_b=\psi''(0)\bg(\int_b\nabla z\cdot a_i\dx
      -\lim_{\eps\to0}\frac{1}{V}\int_{\omega_b\setminus
      \bar{C}_\eps}\nabla z\cdot a_i\dx\bg)  +O\b(|Dz_b|^3\b).
  \end{equation*}
  Since $\nabla z = \sum_{(C,s)\in\D}\nabla\yh\circ G^C+\nabla\ycorr$, the
  only change to the analysis carried out in the proof of
  Lemma~\ref{th:residual_sum_yhat} is to estimate the terms involving
  $\nabla\ycorr$.
  As $\int_b \nabla \ycorr\cdot a_i\dx = \smfrac{1}{|\omega_b|}
  \int_{\omega_b} \nabla I\ycorr\cdot a_i\dx$, applying Jensen's
  inequality and standard interpolation error estimates (see for example
  \S4.4 of \cite{BrennerScott}) gives
  \begin{align*}
    \int_b\nabla \ycorr\cdot a_i\dx -\frac{1}{V}\!\int_{\omega_b}
      \nabla \ycorr\cdot a_i\dx &= \frac{1}{V}
      \int_{\omega_b} \b(\nabla I\ycorr-\nabla \ycorr\b)
      \cdot a_i\dx\leq \frac{1}{\sqrt V}\b\|\nabla I\ycorr-\nabla \ycorr
      \b\|_{\LL^2(\omega_b)}\leq c\b\|\nabla^2 \ycorr\b\|_{\LL^2(\omega_b)}
  \end{align*}
  where $c>0$ is a fixed constant. Applying Young's inequality and
  \eqref{eq:Dycorr_bnd} to estimate $|Dz_b|^3$ now leads immediately to
  \eqref{eq:fin:int_stress_est}.

  Estimate \eqref{eq:fin:bdry_stress_est} follows in a similar way: 
  Taylor expanding $g_b$, but noting that $|\omega_b|=V/2$ and $\omega_b$ is
  no longer symmetric, the same argument used above gives
  \begin{equation*}
    |g_b+\Sigma_b| \leq\int_{b}\bg|\smfrac12\nabla z\cdot a_i-
      \int_{\gamma_\zeta^s}\nabla z\cdot\nu\dt\bg|\ds
      +c\|\nabla^2\ycorr\|_{\LL^2(\omega_b)}+\sum_{(C,s)\in\D}
      \b\|\nabla^2\yh\circ G^C\b\|_{\LL^\infty(\omega_b)}+O(|Dz_b|^3).
  \end{equation*}
  Applying \eqref{eq:Dycorr_bnd} to the first and last terms and
  and using the decay of $\nabla\yh$ now yields
  \begin{equation*}
    |g_b+\Sigma_b|\leq c\b(1+\mathcal{H}^1(P_\zeta)\b)\sum_{(C,s)\in\D}\b(1+
      \dist(P_\zeta,C)\b)^{-1}+c\|\nabla^2\ycorr\|_{\LL^2(\omega_b)}.
  \end{equation*}
  Upon recalling the definition of $\ind(\partial W^\Om)$ from
  \eqref{eq:ind_defn}, the proof is complete.
\end{proof}

We can now deduce a residual estimate for the predictor in the
finite domain case.

\begin{lemma}
\label{th:fin_latt_cons}
Suppose $\Om$ is a convex lattice polygon, and $z$ is the approximate
solution corresponding to a dislocation configuration $\D$ in $\Om$
defined in \eqref{eq:appr_sol_fin} with truncation radius
$R=\min\b\{L_\D/5,S_\D^{1/2}\b\}$.  Then there exist constants $L_0$,
$S_0$ and $c$ depending only on $N=|\D|$ and $\ind(\partial W^\Om)$
such that whenever $L_\D\geq L_0$ and $S_\D\geq S_0$,
  \begin{equation*}
    \b\|\del E^{\Om}(z)\b\|_{(\Hsi(\Om))^*}\leq c\B(L_\D^{-1}+S_\D^{-1/2}\B).
  \end{equation*}
\end{lemma}

\begin{proof}
We begin by enumerating the elements $(C^i,s^i)\in\D$, and
set $G^i:=G^{C^i}$. For $i=1,\ldots,N$, we let $\yh^i = \yh\circ G^i$,
let $y^i=(\yh+\Pi_Ru)\circ G^i$, and let $\ycorr^i$ be the corrector solving
\eqref{eq:y_corr_prob_1disl} with $x'=x^{C^i}$.

Define $r:=2(R+1)=2\b(\!\min\b\{L_\D/5,S_\D^{1/2}\b\}+1\b)$. Taking a test
function $v\in\Hsi(\Om)$, let
\begin{equation*}
  v^i(\xi):=\Pi_r^{C^i}v(\xi)\qquad\text{and}\qquad
    v^0(\xi):=v(\xi)-\sum_iv^i(\xi).
\end{equation*}
Lemma \ref{th:trunc_est} implies there is a universal constant independent
of $\Om$ such that $\|Dv^i\|_2\leq C\|Dv\|_2$ for any $i=0,\ldots,N$.
Adding and subtracting terms, we write
\begin{align}
  \<\del E^{\Om}(z),v\>&=\<\del E^{\Om}(z),v^0\>
    +\sum_i\b\<\b[\del E^\Om(z)-\del E^\L(y^i)\b],
    v^i\b\>\notag\\
  &\qquad\qquad+\sum_i\b\<\b[\del E^\L(y^i)
    -\del E^\L\b(\yh^i+ u\circ G^i\b)\b],v^i\b\>,\notag\\
  &=:\mathrm{T}_1+\mathrm{T}_2+\mathrm{T}_3.\label{eq:fin:pred_prf:decomp}
\end{align}
We estimate each of these terms in turn.

{\it The term ${\rm T}_1$: } Applying Lemma \ref{th:residual_sum_ycorr} and
the fact that $z=\sum_{i=1}^N\yh^i+\ycorr^i$ in $\supp(v^0)$, we
make a similar estimate to that in Lemma \ref{th:full_latt_cons}:
\begin{align}
  \b|\mathrm{T}_1\b|&= \bg|\sum_{b\in\BOm}g_b Dv^0_b
    +\sum_{b\in\partial W^\Om}\Sigma_bDv^0_b\bg|\notag\\
  &\leq c_1\Bg(\bg(\sum_{\substack{(C,s)\in\D,\,b\in\BOm\\\dist(b,C)
    \geq r/2-1}} \b(1+\dist(b,C)\b)^{-6}\bg)^{1/2}
    +\|\nabla^2\ycorr\|_{\LL^2(W^\Om)}\Bg) \|Dv^0\|_2\notag\\
  &\qquad\quad+c_2\bg(\sum_{\zeta\in\partial W^\Om\cap\partial U^\Om}\b(1+\dist(P_\zeta,C^i)\b)^{-2}\bg)^{1/2}\|Dv^0\|_2,\notag\\
    &\leq c\B(r^{-2}+S_\D^{-1}\log(S_\D)+\ind\b(\partial W^\Om\b)^{1/2}S_\D^{-1/2}\B)\|Dv^0\|_2.
    \label{eq:fin:predprf:T1}
\end{align}
To arrive at the final line we have used \eqref{eq:D2ycorr_bnd}, and the
constant $c$ here is independent of the domain and the index.

{\it The term ${\rm T}_2$: }
For the second set of terms, we have $z-y^i=\sum_{j\neq i}\yh^j
+\sum_j\ycorr^j$ in the support of $v^i$.
We expand as in Lemma \ref{th:full_latt_cons} to obtain
\begin{multline}
  \b\<\del E^\Om(z)-\del E\b(y^i\b),v^i\b\>
    =\sum_{b\in\BOm}\psi''(s_b)\bg(\sum_{j\neq i} D\yh^j_b
    +\sum_{j=0}^N D\ycorr^j_b\bg)Dv^i_b,\\
  =\psi''(0)\sum_{b\in\BOm}\bg(\sum_{j\neq i} D\yh^j_b+
    D\ycorr^j_b\bg)Dv^i_b +\psi''(0)\sum_{b\in\BOm} D\ycorr^i_bDv^i_b
    +\sum_{b\in\BOm}h_bDv^i_b,\label{eq:fin:predprf:T2taylor}
\end{multline}
where $|s_b|\lesssim\sum_j(1+\dist(b,C^j))^{-1}+S_\D^{-1}$ and a 
Taylor expansion yields
\begin{equation*}
  |h_b| = \B|\b(\psi''(s_b) - \psi''(0)\b)\B(\sum_{j \neq i} D
    \yh^j_b+\sum_{j=0}^N D\ycorr^j_b\B)\B| \lesssim |s_b|^2\,r^{-1}.
\end{equation*}
Applying Lemma \ref{th:residual_sum_ycorr} to the first term in
\eqref{eq:fin:predprf:T2taylor}, a similar argument to that used to arrive
at \eqref{eq:fin:predprf:T1} gives
\begin{equation*}
  \sum_{b\in\BOm}\bg(\sum_{j\neq i} D\yh^j_b+D\ycorr^j_b\bg)Dv^i_b
    \leq c\B(r^{-2}+S_\D^{-1}\log(S_\D)\B)\|Dv^i\|_2.
\end{equation*}
Applying the global form of \eqref{eq:Dycorr_bnd} to the second term in
\eqref{eq:fin:predprf:T2taylor},
\begin{equation*}
  \sum_{b\in\BOm}D\ycorr^i_bDv^i_b\leq c_1\,rS_\D^{-1}\|Dv^i\|_2,
\end{equation*}
and finally,
\begin{equation*}
  \sum_{b\in\BOm}h_bDv^i_b\leq r^{-1}\Bg(\bg(
    \sum_{\substack{b\in\BOm,\,(C,s)\in\D\\\dist(b,C^i)\leq r+1}}
    \b(1+\dist(b,C)\b)^{-4}\bg)^{1/2}+rS_\D^{-2}\Bg)\|Dv^i\|_2\leq
    c\B(r^{-1}+S_\D^{-2}\B)\|Dv^i\|_2.
\end{equation*}
Combining these estimates gives
\begin{equation}
  \<\del E^\Om(z)-\del E(y^i),v^i\>
    \leq c\b(r^{-1}+S_\D^{-1}\log(S_\D)+r S_\D^{-1}\b)\|Dv^i\|_2.
    \label{eq:fin:predprf:T2}
\end{equation}

{\it The term ${\rm T}_3$: }
The final group may be once more estimated using the truncation result of
Lemma \ref{th:trunc_est}, giving
\begin{equation}
  \b|\b\<\del E(y^i)-\del E(\yh^i+u\circ G^i),v^i\b\>\b|\lesssim
    R^{-1}\|Dv^i\|_2.\label{eq:fin:predprf:T3}
\end{equation}

{\it Conclusion: } Inserting the estimates \eqref{eq:fin:predprf:T1},
\eqref{eq:fin:predprf:T2} and \eqref{eq:fin:predprf:T3} into
\eqref{eq:fin:pred_prf:decomp}, and using the fact that
$\|Dv^i\|_2\lesssim \|Dv\|_2$, we obtain the bound
\begin{equation*}
  \b|\<\del E^\Om(z),v\>\b|\lesssim \B(L_\D^{-1}+S_\D^{-1/2}\B)
   \|Dv\|_2.\qedhere
\end{equation*}
\end{proof}

\subsection{Stability of the predictor}

Next we prove the stability of the predictor configuration defined in
\eqref{eq:appr_sol_fin}.

\begin{lemma}
\label{th:fin_latt_stab}
Given $I_0$ and $N\in\N$, there exist $R_0=R_0(N)$, $L_0=L_0(N)$ and
$S_0=S_0(N,I_0)$ such that whenever $z$ is the approximate solution 
corresponding to a dislocation configuration $\D$ in a convex lattice polygon
$\Om$ with truncation radius $R$ given in \eqref{eq:appr_sol_fin}, and
furthermore:
\begin{enumerate}
  \item $\ind(\partial W^\Om)\leq I_0$,
  \item $S_\D\geq S_0$, $L_\D\geq L_0$ and $R\geq R_0$,
\end{enumerate}
then there exists $\lambda\geq\ld/2$ such that
\begin{equation*}
  \<\ddel E^\Om(z)v,v\>\geq \lambda\|Dv\|_2^2\qquad\text{for all}\quad v\in\Hsi(\Om).
\end{equation*}
\end{lemma}
\begin{proof}
Fixing $I_0$ and $N$, we choose $R_0$ and $L_0$ such that the conclusion of
Lemma \ref{th:full_latt_stab} holds for any dislocation configuration $\D$ in
$\L$ with $|\D|=N$. Throughout the proof, we fix $R$ to be any number
with $R\geq R_0$, and we will consider only configurations such that
$L_\D\geq L_0$.

Suppose for contradiction that there exists a sequence of domains $\Om^n$
with accompanying dislocation configurations $\D^n$ which together satisfy
\begin{enumerate}
  \item $\ind(\partial W^{\Om^n})=I_0$,
  \item $N:=|\D^n|$, $|\{(C,+1)\in\D^n\}|$ and $|\{(C,-1)\in\D^n\}|$ are
    constant,
  \item $(C_0,+1)\in\D^n$,
  \item $S^n:=S_{\D^n}\to\infty$ as $n\to\infty$ and
  \item $\ddel E^{\Om^n}(z^n)<\ld/2$ for all $n$, where 
    \begin{equation*}
      z^n:=\sum_{(C,s)\in\D^n}s(\yh+\Pi_Ru)\circ G^C+\ycorr^n,
    \end{equation*}
    and $\ycorr^n$ solves \eqref{eq:y_corr_prob} with $\Om=\Om^n$.
\end{enumerate}
We note that condition (3) may be assumed without loss of generality by
applying lattice symmetries. Condition (5) implies that there exists
$v^n\in\Hsi(\Om^n)$ such that $\|Dv^n\|_2=1$ and
\begin{equation*}
  \lambda^n:=\inf_{\substack{v\in\Hsi(\Om^n)\\\|Dv\|_2=1}}
    \<\ddel E^{\Om^n}(z^n)v,v\>=\<\ddel E^{\Om^n}(z^n)v^n,v^n\><\ld/2,
\end{equation*}
since this is a minimisation problem for a continuous function over a compact
set.

For each $n$, enumerate $(C^{n,i},s^{n,i})\in\D^n$, and let
$G^{n,i}:= G^{C^{n,i}}$ and $H^{n,i}:= H^{C^{n,i}}$. Considering $Dv^n$ as
an element of $\ell^2(\Bonds)$ by extending
\begin{equation*}
  Dv^n_b:=\cases{
    Dv^n_b & b\in\BOm,\\
    0      & b\in\Bonds\setminus\BOm,
  }
\end{equation*}
there exists a subsequence such that $Dv^n\circ H^{n,i}$ is weakly convergent
for each $i$. For given $i$ and $j$, $\dist(C^{n,i},C^{n,j})$ either 
remains bounded or tends to infinity, and so define an equivalence relation
$i\sim j$ if and only if $\dist(C^{n,i},C^{n,j})$ is uniformly bounded as
$n\to\infty$.

By possibly taking further subsequences, we may assume that if $i\sim j$ then
$Q^{ji}:=G^{n,j}\circ H^{n,i}$ is constant along the sequence, and hence if
$Dv^n\circ H^{n,i}\wto D\bar{v}^i$ for each $i$,
\begin{equation*}
  D\bar{v}^j\circ Q^{ji} = D\bar{v}^i\qquad\text{when}
    \quad i\sim j.
\end{equation*}
For each equivalence class, $\lbrack i\rbrack$, define
\begin{equation*}
  y^{n,\lbrack i \rbrack}:= \sum_{j\in\lbrack i\rbrack}s^j(\yh+\Pi_Ru)
   \circ G^{n,j}.
\end{equation*}
Using the result of \cite[Lemma 4.9]{EOS13}, there exists
a sequence $r^n\to\infty$ which we may also assume satisfies
\begin{equation*}
 r^n\leq \min_{i\nsim j}
  \b\{\dist(C^{n,i},C^{n,j})\b\}/5\quad\text{and}\quad r^n\leq S^n/5,
\end{equation*}
so that, defining $w^{n,\lbrack i\rbrack} := \Pi_{r^n}^{C^{n,i}}v^n$,
\begin{equation*}
w^{n,\lbrack i\rbrack}\circ H^{n,i}\to\bar{w}^{\lbrack i\rbrack}\text{ in }
\Hsi(\L)\qquad\text{and}\qquad(Dv^n-Dw^{n,\lbrack i\rbrack})\circ H^{n,i}\wto 0
\text{ in }\ell^2(\Bonds),
\end{equation*} where $i$ is a fixed representative of $[i]$. Further defining
$Dw^{n,0}:=Dv^n-\sum_{\lbrack i \rbrack} Dw^{n,[i]}$, we have
\begin{equation*}
\<\ddel E^{\Om^n}(z^n)v^n,v^n\>=\<\ddel E^{\Om^n}(z^n)w^{n,0},w^{n,0}\>+
\sum_{\lbrack i\rbrack}\B(2\<\ddel E^{\Om^n}(z^n)w^{n,\lbrack i\rbrack},
  w^{n,0}\>+\<\ddel E^{\Om^n}(z^n)w^{n,\lbrack i\rbrack},
  w^{n,\lbrack i\rbrack}\>\B).
\end{equation*}
The definition of $r^n$ and \eqref{eq:Dycorr_bnd} imply that
$\|\nabla\ycorr^n\|_{\LL^\infty(U^{\Om^n})}\leq c_1/r^n$, where $c_1$ is 
independent of $n$, so in a similar fashion to the proof of 
Lemma~\ref{th:full_latt_stab}, we obtain:
\begin{align*}
  \b\<\ddel E^{\Om^n}(z^n) w^{n,0},w^{n,0}\b\>&=
    \b\<[\ddel E^{\Om^n}(z^n)-\ddel E^{\Om^n}(0)]w^{n,0},w^{n,0}\b\>
    +\b\<\ddel E^{\Om^n}(0) w^{n,0},w^{n,0}\b\>\\
  &\geq \b(\psi''(0)-c/r^n\b)\b\|Dw^{n,0}\b\|_2^2,\\
  \b\<\ddel E^{\Om^n}(z^n) w^{n,\lbrack i\rbrack},w^{n,\lbrack i\rbrack}\b\>&=
  \b\<[\ddel E^{\Om^n}(z^n)-\ddel E^{\L}(y^{n,\lbrack i\rbrack})]
    w^{n,\lbrack i\rbrack},w^{n,\lbrack i\rbrack}\b\>+
    \b\<\ddel E^{\L}(y^{n,\lbrack i\rbrack}) w^{n,\lbrack i\rbrack},
    w^{n,\lbrack i\rbrack}\b\>,\\
  &\geq \b(\lLR-c/r^n\b)\b\|Dw^{n,\lbrack i\rbrack}\b\|_2^2,
    \qquad\text{and}\\
  \b\<\ddel E^{\Om^n}(z^n) w^{n,0},w^{n,\lbrack i\rbrack}\b\>&\to0\qquad
    \text{as}\quad n\to\infty,
\end{align*}
where $c$ represents a constant independent of $n$.
Furthermore, the arguments of the proof of Lemma \ref{th:full_latt_stab}
imply that
\begin{equation*}
  \liminf_{n\to\infty}\B(\sum_i \|Dw^{n,\lbrack i\rbrack}\|_2^2
     -\|Dv^n\|_2^2\B)\geq0,
\end{equation*}
and so we deduce that
\begin{equation*}
  \lambda_n = \b\<\ddel E^{\Om^n}(z^n)v^n,v^n\b\>\geq \ld/2>0
\end{equation*}
for $n$ sufficiently large, providing the required contradiction.
\end{proof}

\subsection{Conclusion of the proof of Theorem \ref{th:full_latt},
  Convex Lattice Polygon Case}
To conclude the proof of conclusions (1) and (2) of Theorem
\ref{th:full_latt}, we may apply small modifications of the arguments
used in \S\ref{sec:inf:conclusion1} and \S\ref{sec:inf:conclusion1},
and hence we omit these.

To prove conclusion (3), recall the result of Lemma \ref{th:0_glob_stab},
which states that $y\equiv0$ is a globally stable equilibrium in any lattice
domain. When $\Om$ is a convex lattice polygon,
$\Us(\Om)\subset\Hsi(\Om)$, so if $z+w$ is the local equilibrium for $E^\Om$
constructed in (1), then $-z-w\in\Hsi(\Om)$, and furthermore
\begin{equation*}
  E(z+w-z-w;z+w) = E(0;z+w) = - E(z+w;0) < 0,\qquad\text{as}\qquad
    0=\argmin_{u\in\Hsi(\Om)} E(u;0).
\end{equation*}

\bibliographystyle{plain}
\bibliography{qc}

\begin{thebibliography}{10}

\bibitem{ADLGP13}
R.~{Alicandro}, L~{De Luca}, A~{Garroni}, and M~{Ponsiglione}.
\newblock Metastability and dynamics of discrete topological singularities in
  two dimensions: a {$\Gamma$}-convergence approach.
\newblock preprint, 2013.

\bibitem{DislDyn}
R.~J. Amodeo and N.~M. Ghoniem.
\newblock Dislocation dynamics. i. a proposed methodology for deformation
  micromechanics.
\newblock {\em Phys. Rev. B}, 41:6958--6967, 1990.

\bibitem{ArizaOrtiz05}
M.~P. Ariza and M.~Ortiz.
\newblock Discrete crystal elasticity and discrete dislocations in crystals.
\newblock {\em Arch. Ration. Mech. Anal.}, 178(2), 2005.

\bibitem{Bollmann56}
W~Bollmann.
\newblock Interference effects in the electron microscopy of thin crystal
  foils.
\newblock {\em Physical Review}, 103(5):1588, 1956.

\bibitem{BrennerScott}
Susanne~C. Brenner and L.~Ridgway Scott.
\newblock {\em The mathematical theory of finite element methods}, volume~15 of
  {\em Texts in Applied Mathematics}.
\newblock Springer, New York, third edition, 2008.

\bibitem{BulatovCai}
V.~V. Bulatov and W.~Cai.
\newblock {\em Computer Simulations of Dislocations}, volume~3 of {\em Oxford
  Series on Materials Modelling}.
\newblock Oxford University Press, 2006.

\bibitem{EOS13}
V.~Ehrlacher, C.~Ortner, and A.~V. Shapeev.
\newblock Analysis of boundary conditions for crystal defect atomistic
  simulations.
\newblock 2013.

\bibitem{EHIM09}
A.~El~Hajj, H.~Ibrahim, and R.~Monneau.
\newblock Dislocation dynamics: from microscopic models to macroscopic crystal
  plasticity.
\newblock {\em Contin. Mech. Thermodyn.}, 21(2):109--123, 2009.

\bibitem{Fromm93}
Stephen~J. Fromm.
\newblock Potential space estimates for {G}reen potentials in convex domains.
\newblock {\em Proc. Amer. Math. Soc.}, 119(1):225--233, 1993.

\bibitem{GM05}
A.~Garroni and S.~M{\"u}ller.
\newblock {$\Gamma$}-limit of a phase-field model of dislocations.
\newblock {\em SIAM J. Math. Anal.}, 36(6):1943--1964 (electronic), 2005.

\bibitem{GLP10}
Adriana Garroni, Giovanni Leoni, and Marcello Ponsiglione.
\newblock Gradient theory for plasticity via homogenization of discrete
  dislocations.
\newblock {\em J. Eur. Math. Soc. (JEMS)}, 12(5):1231--1266, 2010.

\bibitem{GM06}
Adriana Garroni and Stefan M{\"u}ller.
\newblock A variational model for dislocations in the line tension limit.
\newblock {\em Arch. Ration. Mech. Anal.}, 181(3):535--578, 2006.

\bibitem{GPPS12}
M.~G.~D. Geers, R.~H.~J. Peerlings, M.~A. Peletier, and L.~Scardia.
\newblock Asymptotic behaviour of a pile-up of infinite walls of edge
  dislocations.
\newblock {\em Arch. Ration. Mech. Anal.}, 209(2):495--539, 2013.

\bibitem{Grisvard}
P.~Grisvard.
\newblock {\em Elliptic problems in nonsmooth domains}, volume~69 of {\em
  Classics in Applied Mathematics}.
\newblock SIAM, Philadelphia, PA, 2011.

\bibitem{Hatcher}
Allen Hatcher.
\newblock {\em Algebraic topology}.
\newblock Cambridge University Press, Cambridge, 2002.

\bibitem{HirschHornWhelan56}
PB~Hirsch, RW~Horne, and MJ~Whelan.
\newblock {LXVIII}. {D}irect observations of the arrangement and motion of
  dislocations in aluminium.
\newblock {\em Philosophical Magazine}, 1(7):677--684, 1956.

\bibitem{HirthLothe}
John~Price Hirth and Jens Lothe.
\newblock {\em Theory of Dislocations}.
\newblock Krieger Publishing Company, Malabar, Florida, 1982.

\bibitem{HudsonOrtner13}
T.~{Hudson} and C.~{Ortner}.
\newblock {Existence and stability of a screw dislocation under anti-plane
  deformation}.
\newblock {\em ArXiv e-prints}, March 2013.
\newblock Submitted to Archive for Rational Mechanics and Analysis.

\bibitem{LuskinOrtner13}
Mitchell Luskin and Christoph Ortner.
\newblock Atomistic-to-continuum coupling.
\newblock {\em Acta Numer.}, 22:397--508, 2013.

\bibitem{MP12}
R{\'e}gis Monneau and Stefania Patrizi.
\newblock Homogenization of the {P}eierls-{N}abarro model for dislocation
  dynamics.
\newblock {\em J. Differential Equations}, 253(7):2064--2105, 2012.

\bibitem{Orowan34}
E.~Orowan.
\newblock Zur {K}ristallplastizit{\"a}t. {III}.
\newblock {\em Zeitschrift f{\"u}r Physik}, 89:634--659, 1934.

\bibitem{OrtSha:interp:2012}
C.~Ortner and A.~Shapeev.
\newblock Interpolants of lattice functions for the analysis of
  atomistic/continuum multiscale methods.
\newblock {\em ArXiv e-prints}, 1204.3705, 2012.

\bibitem{Polanyi34}
M.~Polanyi.
\newblock {\"U}ber eine {A}rt {G}itterst{\"o}rung, die einen {K}ristall
  plastisch machen k{\"o}nnte.
\newblock {\em Zeitschrift f{\"u}r Physik}, 89:660--664, 1934.

\bibitem{Ponsiglione07}
Marcello Ponsiglione.
\newblock Elastic energy stored in a crystal induced by screw dislocations:
  from discrete to continuous.
\newblock {\em SIAM J. Math. Anal.}, 39(2), 2007.

\bibitem{SZ10}
Lucia Scardia and Caterina~Ida Zeppieri.
\newblock Line-{T}ension {M}odel for {P}lasticity as the {$\Gamma$}-{L}imit of
  a {N}onlinear {D}islocation {E}nergy.
\newblock {\em SIAM J. Math. Anal.}, 44(4):2372--2400, 2012.

\bibitem{Taylor34}
G.~I. Taylor.
\newblock The mechanism of plastic deformation of crystals. {P}art {I}.
  {T}heoretical.
\newblock {\em Proceedings of the Royal Society of London. Series A, Containing
  Papers of a Mathematical and Physical Character}, 145(855), 1934.

\bibitem{Theil11}
Florian Theil.
\newblock Surface energies in a two-dimensional mass-spring model for crystals.
\newblock {\em ESAIM Math. Model. Numer. Anal.}, 45(5):873--899, 2011.

\bibitem{VCOA07}
R.E. Voskoboinikov, S.J. Chapman, J.R. Ockendon, and D.J. Allwright.
\newblock Continuum and discrete models of dislocation pile-ups. {I}. {P}ile-up
  at a lock.
\newblock {\em Journal of the Mechanics and Physics of Solids}, 55(9):2007 --
  2025, 2007.

\end{thebibliography}

\end{document}